 \documentclass{article}

\usepackage{mathptmx}      
\usepackage{times}
\usepackage{mathtools}
\usepackage{amsmath,amssymb,amsthm}
\usepackage{bbm}
\usepackage{booktabs}
\usepackage{subcaption}
\usepackage{float}
\usepackage[title]{appendix}
\usepackage{color}
\usepackage{comment}
\usepackage[utf8]{inputenc}
\usepackage[english]{babel}
\usepackage{graphicx}
\usepackage[colorinlistoftodos]{todonotes}

\newtheorem{theorem}{Theorem}[section]
\newtheorem{remark}[theorem]{Remark}
\newtheorem{definition}[theorem]{Definition}
\newtheorem{corollary}[theorem]{Corollary}
\newtheorem{lemma}[theorem]{Lemma}
\newtheorem{example}[theorem]{Example}

\usepackage{float}
\usepackage{xcolor}
\usepackage{algorithm}
\usepackage{algpseudocode}
\usepackage[export]{adjustbox}
\usepackage{wrapfig}

\newcommand{\Rb}{\mathbbm{R}}      

\newcommand{\Bc}{\mathcal{B}}

\newcommand{\Qc}{\mathcal{Q}}
\newcommand{\Qb}{\mathbbm{Q}}

\newcommand{\Db}{\mathbbm{D}}
\newcommand{\Eb}{\mathbbm{E}}

\newcommand{\Ib}{\mathbbm{I}}
\newcommand{\Fc}{\mathcal{F}}

\newcommand{\Lc}{\mathcal{L}}

\newcommand{\Pc}{\mathcal{P}}

\newcommand{\Vc}{\mathcal{V}}
\newcommand{\Xc}{\mathcal{X}}
\newcommand{\Yc}{\mathcal{Y}}
\newcommand{\Ic}{\mathcal{I}}
\newcommand{\Wc}{\mathcal{W}}
\newcommand{\Zc}{\mathcal{Z}}

\newcommand{\1}{\mathbbm{1}}

\newcommand{\argmin}{\mathop{\rm argmin}}

\newcommand{\avar}{\text{\rm AVaR}}
\newcommand{\msd}{\text{\rm msd}}

\newcommand{\supp}{\mathop{\rm supp}}

\newcommand{\D}{\mathop{\text{\rm d}\!}}

\newenvironment{tightitemize}{%
    \list{{\textup{$\bullet$}}}{\settowidth\labelwidth{{\textup{\qquad}}}
    \leftmargin\labelwidth \advance\leftmargin\labelsep
    \parsep 0pt plus 1pt minus 1pt \topsep 3pt \itemsep 3pt
    }}{\endlist}

\title{An Integrated Transportation Distance Between Kernels and Approximate
Dynamic Risk Evaluation in Markov Systems\thanks{This work was supported by the National Science Foundation Award DMS-1907522 and by the Office of Naval Research Award N00014-21-1-2161.}}
\author{Zhengqi Lin and Andrzej Ruszczy\'nski}
\date{August 2022}

\begin{document}

\maketitle

\begin{abstract}
We introduce a distance between kernels based on the Wasserstein distances between their values, study its properties, and prove that it is a metric on an appropriately defined space of kernels. We also relate it to various modes
of convergence in the space of kernels. Then we consider the problem of approximating solutions to forward--backward systems, where the forward part is a Markov system described by a sequence of kernels,
and the backward part calculates the values of a risk measure
by operators that may be nonlinear with respect to the system's kernels. We propose to recursively approximate the forward system with the use of the integrated transportation distance between kernels and we estimate the error of the risk evaluation
by the errors of individual kernel approximations. We illustrate the results on stopping problems and several well-known risk measures.
Then we develop a particle-based numerical procedure, in which the approximate kernels have finite support sets. Finally, we illustrate the efficacy of the approach on
the financial problem of pricing an American basket option.

\end{abstract}


\section{Introduction}

We consider a discrete-time Markov system described by the relations:
\begin{equation}
\label{kernel-model}
X_{t+1} \sim Q_t(X_t), \quad t=0,1,\dots,T-1,
\end{equation}
where $X_t\in \Xc$ represents the state at time $t$, $\Xc$ is a Polish space, and $Q_t:\Xc \to \Pc(\Xc)$, $t=0,1,\dots,T-1$, are stochastic kernels (the symbol $\Pc(\Xc)$ denotes the space of probability measures on $\Xc$).
The initial state $X_0=x_0$ is fixed. The model \eqref{kernel-model} is understood as follows: given $X_t=x$, the conditional distribution
of $X_{t+1}$ is $Q_t(x)$. The sequence of kernels $Q_t$, $t=0,\dots,T$, and the distribution of the initial state $\lambda_0$
define a probability measure $P$ on the canonical space $\Xc^{T+1}$. We also consider the filtration $\Fc_t=\Bc(\Xc^{t+1})$, $t=0,\dots,T$.

Suppose a sequence of Borel measurable functions $c_t:\Xc\to\Rb$, $t=0,\dots,T$, is given.
Together with the dynamical system \eqref{kernel-model}, we consider the following backward risk evaluation system
\begin{equation}
\label{DP-risk-finite}
\begin{aligned}
v_t(x) = c_t(x) +  \sigma_t\big(x,Q_t(x), v_{t+1}(\cdot)\big),\quad & x\in \Xc, \quad t=T-1,T-2,\dots,0;\\
v_T(x) = c_T(x), \quad & x \in \Xc.
\end{aligned}
\end{equation}
In equation \eqref{DP-risk-finite}, the operator $\sigma_t:\Xc \times  \Pc(\Xc)\times \Vc\to \Rb$,
where $\Vc$ is a space of Borel measurable real functions on $\Xc$, is a \emph{transition risk mapping}.  Its first argument is the present state~$x$. The second argument is the probability distribution $Q_t(x)$ of the state following $x$ in the system \eqref{kernel-model}. The last argument, the function $v_{t+1}(\cdot)$, is the next state's value: the risk of running the system from the next state
in the time interval from $t+1$ to $T$.  In the next section, we briefly review the background of this backward system
in the dynamic risk theory and provide a more formal definition of the objects involved,
but we want to stress that the
evaluation \eqref{DP-risk-finite} is of relevance for other problems as well.

A simple case of the transition risk mapping is the bilinear form,
\begin{equation}
\label{sigma-E}
\sigma_t\big(x,\mu, v_{t+1}(\cdot)\big) = \Eb_{\mu} \big[ v_{t+1}(\cdot)\big].
\end{equation}
In this case, the
scheme \eqref{DP-risk-finite} evaluates the conditional expectation of the total cost from stage $t$ to the end of the horizon $T$:
\[
v_t(x) = \Eb\big[ c_t(X_t) + \dots + c_T(X_T)\,\big|\, X_t = x \big], \quad x \in \Xc,\quad t=0,\dots,T.
\]
A more interesting application is the \emph{optimal stopping problem}, in which $c_t(\cdot)\equiv 0$, and
\begin{equation}
\label{sigma-stopping}
\sigma_t\big(x,\mu, v_{t+1}(\cdot)\big) = \max\Big( r_t(x)\; ; \; \Eb_{\mu} \big[ v_{t+1}(\cdot)\big]\Big).
\end{equation}
Here, $r_t:\Xc\to \Rb$, $t=0,\dots,T$, represent the rewards collected if the decision to stop at time $t$ and state $x$ is made. Clearly,
with the mappings \eqref{sigma-stopping} used in the scheme \eqref{DP-risk-finite},
\[
v_t(x) = \sup_{{\tau - \text{stopping time}}\atop{t \le \tau \le T}}r_\tau(X_\tau), \quad x \in \Xc,\quad t=0,\dots,T;
\]
see, \emph{e.g.}, \cite{chow}.
The most important difference between \eqref{sigma-E} and \eqref{sigma-stopping} is that the latter is nonlinear with respect to the
probability measure $\mu$.
In the next section, we provide other examples of nonlinear transition risk mappings derived from coherent measures of risk.

One of the challenges associated with the backward system \eqref{DP-risk-finite}
is the numerical solution in the case when the transition risk mappings are nonlinear with respect to the probability measures involved.
The objective of this paper is to present a computational method based on approximating the kernels $Q_t(\cdot)$ by simpler, easier-to-handle kernels $\widetilde{Q}_t(\cdot)$, and using them in the backward system \eqref{DP-risk-finite}. For this purpose, after the
preliminary section, in \S \ref{s:kernel-distance} we introduce the space of kernels under consideration and define a metric on this space.
The metric generalizes the transportation (Wasserstein) metric between probability distributions.
We relate it to various convergence modes in the space of kernels. In \S \ref{s:abstract-method} we describe an iterative scheme
for building the approximate system and we estimate the error of the approximation by the distances of the kernels involved at each stage.
We also illustrate the application of the theory to various specific risk evaluation systems with nonlinear transition risk mappings.
Next, in \S \ref{s:particles}, we specialize
our method by considering kernels supported on finite sets, and we derive tractable linear programming models for minimizing the
approximation error. Finally, in \S \ref{s:illustration}, we illustrate our approach on the problem of evaluating an American basket option.

The problem of approximating stochastic processes in discrete time has attracted the attention of researchers for several decades.
The basic construction is that of a \emph{scenario tree}. In \cite{hoyland2001generating}, the construction of the tree is based on
statistical parameters, such as moments and correlations. A further contribution of \cite{kaut2011shape}  involves copulas to capture the shape of the
distributions. The use of probability metrics to reduce large scenario trees was first proposed in \cite{heitsch2009scenario}. A concept of a distance between stochastic processes
was proposed by \cite{pflug2001scenario}, and used by  \cite{mirkov2007tree,kovacevic2015tree}  to generate scenario trees.
The concept of nested (adapted) distance, using an extension of the Wasserstein metric for processes, was introduced in \cite{pflug2010version} and further developed in \cite{pflug2012distance,pflug2015dynamic}. { Similar ideas are pursued in continuous time in \cite{backhoff2020adapted}.
Ref. \cite{bartl2022sensitivity} addresses the sensitivity of the optimal value of an expected-value problem, when the probability measure perturbation is small in the nested distance.
None of these contributions focuses on Markov systems and the approaches proposed do not reduce to our construction in the Markovian case.}

{ Ref.
\cite{kern2020first} considers perturbations in a transition kernel of a {controlled} Markov system.
The distance between probability kernels defined in \cite[\S 3]{kern2020first} is close to our idea, but it uses the ``sup norm'' over the state space, rather than the ``$\Lc_p$ norm'' in our case (a similar idea appeared earlier in \cite{mirkov2007tree} for scenario trees). This is further used to estimate the error of the
value function in risk-neutral models in \cite{zahle2022concept}. We discuss it in more detail in \S 3 and \S 4.
}

Finally, some recent
contributions focus on mixture models, which are somehow related to our approach, but which measure the distance of mixture distributions rather than kernels.
 The sketched Wasserstein distance, a type of distance metric dedicated to finite mixture models, was proposed in \cite{SW}. Research on Wasserstein-based distances specifically tailored to Gaussian mixture models
 is reported in \cite{GMM1, GMM2, GMM3}.

\section{Preliminaries}
\label{s:preliminaries}

{ In this section, we briefly present the mathematical foundations of the techniques discussed in the paper. In \S \ref{s:2.1}, we summarize the relevant concepts of Markov risk evaluation, and in \S \ref{s:2.2} we recall the basic ideas of the transportation distance between
probability measures.}

\subsection{Markov risk measures}
\label{s:2.1}

A \emph{dynamic risk measure} evaluates the sequence of random costs $Z_t= c_t(X_t)$, $t=0,1,2,\dots,T$,
where $c_t:\Xc\to\Rb$, $t=0,1,\dots,T$, are measurable functions.
 Because of the need to evaluate the risk of the future costs at any time period,  a dynamic measure of risk is a collection of \emph{conditional risk measures} $\rho_{t,T}(Z_t,\dots,Z_T)$, $t=0,\dots,T$. Formally, for $t=0,\dots,T$,
 we consider $\sigma$-subalgebras $\Fc_t = \Bc(\Xc^{t+1})$
 and spaces $\Zc_t$ of $\Fc_t$-measurable real random variables. A conditional risk measure is a functional
 $\rho_{t,T}: \Zc_t \times \cdots \times \Zc_T \to \Zc_t$. We postulate three properties of each conditional risk measure:
 \begin{tightitemize}
 \item[\textbf{Normalization:}] $\rho_{t,T}(0,\dots,0)=0$, $t=0,1,\dots,T$;
\item[\textbf{Monotonicity:}] For every $t=0,\dots,T$, if $Z_s \leq V_s$ for $s=t,\dots,T$, then $\rho_{t,T}(Z_t,\dots,Z_T) \le \rho_{t,T}(V_t,\dots,V_T)$;
\item[\textbf{Translation equivariance:}] $\rho_{t,T}(Z_t,Z_{t+1},\dots,Z_T) = Z_t + \rho_{t,T}(0,Z_{t+1},\dots,Z_T)$, $\forall$ $t=0,\dots,T$.
\end{tightitemize}

 Fundamental for such a
nonlinear dynamic risk evaluation is \emph{time consistency}, discussed in various forms in \cite{ADEHK:2007,CDK:2006,cheridito2011composition,shapiro2021lectures}. We adopt the definition and the following discussion from
\cite{Ruszczynski2010Markov}:
\emph{A dynamic measure of risk is time consistent if  for every $t=0,\dots,{T-1}$, if $Z_t=V_t$ and $\rho_{t+1,T}(Z_{t+1},\dots,Z_T) \le \rho_{t+1,T}(V_{t+1},\dots,V_T)$ a.s., then}
\[
\rho_{t,T}(Z_{t},\dots,Z_T) \le \rho_{t,T}(V_{t},\dots,V_T).
\]
 Such risk measures, under the conditions specified above,
must have a specific recursive form \cite[Thm. 1]{Ruszczynski2010Markov}:
\[
\rho_{t,T}(Z_{t},\dots,Z_T) = Z_t + \rho_t\Big(Z_{t+1}+ \rho_{t+1}\big(Z_{t+2}+ \dots + \rho_{T-1}(Z_T)\cdots\big)\Big),
\]
 where each $\rho_t:\Zc_{t+1}\to \Zc_t$ is a one-step conditional risk measure. This result, generalizing the tower property of conditional expectations, is germane for  our approach.

\emph{Markov risk measures} evaluate the risk of future costs $Z_s= c_s(X_s)$, $s=t,\dots,T$, in a Markov system \eqref{kernel-model}
in such a way that the risk of the future cost sequence is a function of the current state:
\[
\rho_{t,T}(Z_{t},\dots,Z_T) = v_t(X_t).
\]
 This, combined with the properties specified above, implies a very specific structure \cite{fan2018process,bauerle2022markov}:
 transition risk mappings $\sigma_t:\Xc \times  \Pc(\Xc)\times \Vc\to \Rb$, $t=0,\dots,T-1$, exist  such that the risk of each state
 can be evaluated by the procedure \eqref{DP-risk-finite}. Conversely, any collection of transition risk mappings satisfying
 the properties of normalization, monotonicity, and translation equivariance define via \eqref{DP-risk-finite} a time-consistent
  Markov risk measure.

  As mentioned in the introduction, the simplest transition risk mappings are the bilinear forms \eqref{sigma-E}, which lead to the risk-neutral evaluation: the expected value of the sum of the costs. A more interesting example is the \emph{mean--semideviation} mapping
  derived from the corresponding coherent risk measure \cite{ogryczak1999stochastic,ogryczak2001consistency,shapiro2021lectures}:
  \begin{multline}
\label{msd-form}
\msd_p\big(x,\mu,v_{t+1}(\cdot)\big) = \int_\Xc v_{t+1}(y)\;\mu(\D y)\\
+ \varkappa(x) \bigg(\int_\Xc \Big[ v(y) -  \int_\Xc v(y')\;\mu(\D y')\Big]^p_+\;\mu(\D y)\bigg)^{1/p},
\end{multline}
with $p\in [1,\infty)$, and the parameter $\varkappa(x)\in [0,1]$ controlling the degree of risk aversion.

Another example is the Average Value at Risk \cite{rockafellar2000optimization,ogryczak2002dual,shapiro2021lectures}:
\begin{equation}
\label{avar-def}
\avar_\alpha\big(x,\mu,v_{t+1}(\cdot)\big)  = \inf_{\eta \in \Rb} \Big\{ \eta + \frac{1}{\alpha}
\Eb_{\mu}\big[\max(0,v_{t+1}(\cdot)-\eta)\big]\Big\},\quad \alpha \in (0,1].
\end{equation}
Usually, it does not occur alone, but rather in mixtures, as in \emph{spectral} measures (see, \emph{e.g.}, \cite{pflug2007modeling,shapiro2021lectures})
\begin{equation}
\label{spectral}
\sigma_t\big(x,\mu,v_{t+1}(\cdot)\big) = \int_0^1 \avar_\alpha\big(x,\mu,v_{t+1}(\cdot)\big) \;\theta(\D\alpha),
\end{equation}
where $\theta$ is a probability measure on $(0,1]$.

Summing up, the risk evaluation procedure \eqref{DP-risk-finite} is not an arbitrary construction,
but rather the result of
assumptions of normalization, monotonicity, translation, time consistency, and the Markov property.
The transition risk mappings are nonlinear operators with respect to the probability measure, and the
numerical evaluation of risk is a difficult task. Structures of the form \eqref{DP-risk-finite} arise also in the discretization of backward stochastic differential
equations \cite{ruszczynski2020dual}.
For recent applications of Markov risk measures in the control of dynamical systems, see \cite{majumdar2020should,sopasakis2019risk,kose2021risk}.

\subsection{The Wasserstein distance}
\label{s:2.2}

Another essential ingredient of our construction is the Wasserstein distance between measures. As before,
$\Xc$ is a Polish space, with the metric $d(\cdot,\cdot)$, and the associated Borel $\sigma$-field $\Bc(\Xc)$.
The symbol $\Pc(\Xc)$ denotes the space of probability measures on $\Bc(\Xc)$.  For $p \ge 1$, we consider the
space
    \[
\Pc_{p}(\mathcal{X}):=\left\{\mu \in \Pc(\mathcal{X}) : \ \int_{\mathcal{X}} d\left(x_{0}, x\right)^{p} \;\mu(d x)<+\infty\right\},
\]
where $x_{0} \in \mathcal{X}$ is arbitrary.
In the brief summary below, we follow \cite{villani2009optimal}. The reader is referred to this monograph, as well as to \cite{rachev1998mass},
for an extensive exposition and historical account.

\begin{definition} \label{Wass_d1}
The Wasserstein distance of order $p \in[1, \infty)$ between two probability measures $\mu, \nu\in \Pc_p(\Xc)$ is defined by the formula
\begin{equation} \label{Wass}
W_{p}(\mu, \nu) =\left(\inf _{\pi \in \Pi(\mu, \nu)} \int_{\mathcal{X} \times \mathcal{X} } d(x, y)^{p} \;  \pi({\D x}, {\D y})\right)^{1 / p} ,
\end{equation}
where $\Pi(\mu, \nu)$ is the set of all probability measures in $\Pc_p(\Xc\times\Xc)$ with the marginals $\mu$ and $\nu$.
The measure $\pi^* \in \Pi(\mu, \nu)$ that realizes the infimum in Eq. \eqref{Wass} is called the \emph{optimal coupling} or the \emph{optimal transport plan}.
\end{definition}
For each $p\in [1,\infty)$, the function $W_{p}(\cdot,\cdot)$ defines a metric on $\Pc_{p}(\mathcal{X})$. Furthermore, for all $\mu,\nu\in \Pc_{p}(\mathcal{X})$ the optimal coupling realizing the infimum in \eqref{Wass} exists.
From now on, the space $\Pc_{p}(\mathcal{X})$ will be always equipped with the distance $W_p(\cdot,\cdot)$.
\begin{remark}
\label{r:discrete}
{\rm
Problem \eqref{Wass} has a convenient linear programming representation for discrete measures.
Let $\mu$ and $\nu$ be discrete measures in $\Pc(\mathcal{X})$, supported at positions $\{x^{(i)}\}_{i=1}^{N}$ and $\{z^{(s)}\}_{s=1}^{S}$ with normalized (totaling 1) positive weight vectors $w_{x}$ and $w_{z}$:
$\mu=\sum_{i=1}^{N} w_{x}^{(i)} \delta_{x^{(i)}}$, $\nu=\sum_{s=1}^{S} w_{z}^{(s)} \delta_{z^{(s)}}$.
For $p \geq 1$, let $D \in {R}_{+}^{N \times S}$ be the distance matrix defined as $D_{i s}=d(x^{(i)},z^{(s)})^{p_{\phantom{|}}}$. Then the $p$th power of the $p$-Wasserstein distance between the measures $\mu$ and $\nu$ is the optimal value of the following transportation problem:
\begin{equation} \label{12}
\min _{\pi \in {R}_{+}^{N \times S}}\sum_{i s}  D_{i s} \pi_{i s}\quad
\text { s.t.} \quad  \pi^\top \1_{N}=w_{x}, \
 \ \;\; \pi \1_{S}=w_{z}.
\end{equation}
 Its regularized version can be efficiently solved { with almost
 linear complexity with respect to $NS$};  see \cite{sinkhorn1,Greenkhorn}}.
\end{remark}

The following classical result, known as the Kantorovich–Rubinstein duality \cite{kantorovich1958space}, provides an alternative characterization of $W_1(\cdot,\cdot)$.
\begin{theorem}
\label{t:KR}
For any $\mu, \nu$ in $\Pc_{1}(\mathcal{X})$,
\begin{equation}\label{dual_wass}
W_{1}(\mu, \nu)=\sup _{\|\psi\|_{\text {\rm Lip}} \leq 1}\left\{\int_{\mathcal{X}} \psi(x)\; \mu({\D x})-\int_{\mathcal{X}} \psi(x) \; \nu({\D x})\right\},
\end{equation}
where $\|\psi\|_{\text{\rm Lip}}$ denotes the minimal Lipschitz constant of the function $\psi:\Xc\to \Rb$.
\end{theorem}
In the discrete case, it follows from the linear programming duality for problem \eqref{12}.


We now briefly review the convergence concepts in the space $\Pc_p(\Xc)$.
The notation $\mu_{k} \rightharpoonup \mu$ means that $\mu_{k}$ converges weakly to $\mu$, i.e. $\int \varphi(x)\; \mu_{k}({\D x}) \rightarrow \int \varphi(x) \;\emph{} \mu({\D x})$ for all bounded continuous functions $\varphi:\Xc \to \Rb$.

\begin{definition}
\label{d:villani-def-1}
    Let $(\mathcal{X}, d)$ be a Polish space, and $p \in[1, \infty)$. Let $\left\{\mu_{k}\right\}_{k \in {N}}$ be a sequence of probability measures in $\Pc_{p}(\Xc)$ and let $\mu$ be an element of $\Pc_{p}(\mathcal{X})$. Then $\left\{\mu_{k}\right\}$  is said to converge to $\mu$ weakly in $\Pc_{p}(\mathcal{X})$, written $\mu_k \overset{p}\to \mu$,
   if for some (and then any) $x_{0} \in \mathcal{X}$, and
for all continuous functions $\varphi$ with $|\varphi(x)| \leq 1+d\left(x_{0}, x\right)^{p}$ one has
\begin{equation}
    \int \varphi(x) \; \mu_{k}({\D x}) \longrightarrow \int \varphi(x) \; \mu({\D x}).
\end{equation}
\end{definition}

The fundamental property of the Wasserstein distance $W_p(\cdot,\cdot)$ is that it metricizes the topology of weak
convergence in $\Pc_p(\Xc)$.
\begin{theorem}
\label{t:metrization}
Let $(\Xc,d)$  be a Polish space, $p\in [1,\infty)$; then
$\mu_k \overset{p}\to \mu$ if and only if
 $W_p(\mu_k ,\mu) \to 0$.
Furthermore,  $\big(\Pc_{p}(\mathcal{X}),W_{p}\big)$ is a Polish space.
\end{theorem}
By the triangle inequality, $W_{p}(\cdot,\cdot)$ is continuous on
$\Pc_{p}(\mathcal{X})\times \Pc_{p}(\mathcal{X})$.

\section{The Integrated Transportation Distance Between Kernels}
\label{s:kernel-distance}

We now introduce an essential concept in our research: the integrated transportation distance between kernels.

Suppose $\Xc$ and $\Yc$ are Polish spaces.  By the measure disintegration formula, every probability measure $\mu\in \Pc(\Xc\times\Yc)$ admits a disintegration
$\mu = {\lambda} \circledast Q$, where ${\lambda} \in \Pc(\Xc)$ is the marginal distribution on $\Xc$, and $Q:\Xc \to \Pc(\Yc)$ is a \emph{kernel} (a function
such that for each $B\in \Bc(\Yc)$ the mapping $x\mapsto Q(B|x)$ is Borel measurable):
\[
\mu(A \times B) = \int_A  Q(B|x)\;{\lambda}({\D x}), \quad \forall \big(A\in \Bc(\Xc)\big),\;\forall \big(B \in \Bc(\Yc)\big).
\]
Conversely, given a marginal ${\lambda} \in \Pc(\Xc)$  and a kernel $Q:\Xc \to \Pc(\Yc)$, the above formula defines a probability measure ${\lambda} \circledast Q$
on $\Xc \times \Yc$.
Its marginal on $\Yc$ is the \emph{mixture distribution} $\lambda \circ Q$ given by
\[
(\lambda \circ Q)(B) = \int_\Xc  Q(B|x)\;{\lambda}({\D x}), \quad \forall \, B \in \Bc(\Yc).
\]
We intend to define a distance between kernels with the use of the Wasserstein metric in the space of probability measures.
To this end, we restrict the class of kernels under consideration. We use the same symbol $d(\cdot,\cdot)$ to denote the metrics on $\Xc$ and $\Yc$; the space will be clear from the context.

\begin{definition}
\label{d:kernel-space}
The kernel space of order $p\in [1,\infty)$  is the set
\begin{multline}
\label{kernel-class}
\Qc_p(\Xc,\Yc) = \Big\{ Q:\Xc \to \Pc_p(\Yc) : \forall \big(B\in \Bc(\Yc)\big) \; Q(B|\cdot) \text{ is Borel measurable},\\
\exists (C>0)\,{  \forall(x\in \Xc)} \int_\Yc d(y,y_0)^p\; Q({\D y}|x) \le C\big(1 + d(x,x_0)^p\big)\Big\}.
\end{multline}
\end{definition}
It is evident that the choice of the points $x_0\in \Xc$ and $y_0\in \Yc$ is irrelevant in this definition.

\begin{definition}
\label{d:kernel-distance}
The integrated transportation distance of degree $p$ between two kernels $Q$ and $\widetilde{Q}$ in $\Qc_p(\Xc,\Yc)$ with fixed marginal ${\lambda}\in \Pc_p(\Xc)$ is defined as
\begin{equation} \label{TS}
    \Wc_{p}^\lambda(Q, \widetilde{Q})=\left(\int_{\Xc} \big[{W}_{p}(Q(\cdot | x), \widetilde{Q}(\cdot | x))\big]^p \;{\lambda}(\D x)\right)^{1/p} .
\end{equation}
\end{definition}
From now on, for a fixed marginal $\lambda\in \Pc_p(\Xc)$, we shall identify the kernels $Q$ and $\widetilde{Q}$
if ${W}_{p}(Q(\cdot | x), \widetilde{Q}(\cdot | x))=0$ for $\lambda$-almost all $x\in \Xc$. Thus, we consider the
space $\Qc_p^\lambda(\Xc,\Yc)$ of equivalence classes of $\Qc_p(\Xc,\Yc)$.

\begin{theorem} \label{TS_T}
For any $p\in [1,\infty)$ and any $\lambda\in \Pc_p(\Xc)$, the function $\Wc_p^\lambda(\cdot,\cdot)$, defines a  metric on the space $\Qc_p^\lambda(\Xc,\Yc)$.
\end{theorem}
{  The proof is provided in the Appendix.}

{
\begin{remark}
    \label{r:kern}
{\rm Our construction of the kernel space \eqref{kernel-class} and the metric  \eqref{TS} are related to the ideas used in \cite{mirkov2007tree} for scenario trees, and
refined in \cite[\S 3]{kern2020first} for Markov systems. In our notation, the authors of \cite{kern2020first}
propose the metric
\[
\Db_{p}(Q, \widetilde{Q})=\sup_{x\in \Xc} \frac{1}{\psi(x)}{W}_{p}(Q(\cdot | x), \widetilde{Q}(\cdot | x)), \]
with a gauge function $\psi:\Xc\to [1,\infty).$ If $\psi(\cdot)\equiv 1$ we have
$\Wc_{p}^\lambda(Q, \widetilde{Q}) \le  \Db_{p}(Q, \widetilde{Q}) $.
The uniformity (relative to the gauge function)
of the approximation over all states $x\in \Xc$ is most suitable for situations when nothing is known about the
distribution of $x$. In our approximation method in the next section,
the marginal $\lambda$ is not arbitrary, but
it closely approximates the marginal distribution of the state in the original system. Thanks to that, the use of the metric
\eqref{TS} allows for controlling the propagation of errors in the backward system \eqref{DP-risk-finite}.
It also eliminates the need to work with gauge functions in unbounded spaces.
}
\end{remark}
}

For a kernel $Q \in \Qc_p(\Xc,\Yc)$, and every $\lambda \in \Pc_p(\Xc)$ the measure $\lambda \circ Q$ is an element of $\Pc_p(\Yc)$, because
\begin{multline*}
\int_\Yc d(y,y_0)^p\; (\lambda\circ Q)({\D y}) = \int_\Xc \int_\Yc d(y,y_0)^p\;\; Q({\D y}|x)\;\lambda({\D x}) \\
 \le C(Q) \int_\Xc  \big(1 + d(x,x_0)^p\big)\;\lambda({\D x}) < \infty.
\end{multline*}
In a similar way, the measure ${\lambda} \circledast Q \in \Pc_p(\Xc\times\Yc)$, because
\begin{multline*}
\int_\Xc\int_\Yc \big[d(x,x_0)^p+d(y,y_0)^p\big]\;  Q({\D y}|x)\;\lambda({\D x}) \\
= \int_\Xc \bigg[d(x,x_0)^p+ \int_\Yc d(y,y_0)^p\;\; Q({\D y}|x)\bigg]\;\lambda({\D x}) \\
 \le (C(Q)+1) \int_\Xc  \big(1 + d(x,x_0)^p\big)\;\lambda({\D x}) < \infty.
\end{multline*}

The integrated transportation distance provides an upper bound on the distances between two mixture distributions and between two composition distributions.

\begin{theorem} \label{TS_T2}
For all $\lambda \in \Pc_p(\Xc)$ and all $Q ,  \widetilde{Q} \in \Qc_p^\lambda(\Xc,\Yc)$,
\begin{equation}
\label{hierarchy}
\Wc_p^\lambda(Q,\widetilde{Q}) \geq {W}_p({\lambda} \circledast Q,{\lambda} \circledast \widetilde{Q}) \geq
{W}_p({\lambda} \circ Q,{\lambda} \circ \widetilde{Q}).
\end{equation}
\end{theorem}
{  The proof is provided in the Appendix.}

The inequalities in Theorem \ref{TS_T2} may be strict, as illustrated in the example of
$\Xc=\{0,\varepsilon\}$ with $\varepsilon \in (0,1)$, $\Yc=\{0,1\}$, $\lambda = (1/2,1/2)$, $Q(\cdot|x) = \delta_{\{\text{\rm sign}(x)\}}$,
and $\widetilde{Q}(\cdot|x) = \delta_{\{1-\text{\rm sign}(x)\}}$, in which $\Wc_{p}^\lambda(Q, \widetilde{Q})=1$, ${W}_p({\lambda} \circ Q,{\lambda} \circ \widetilde{Q})=0$,
and ${W}_p({\lambda} \circledast Q,{\lambda} \circledast \widetilde{Q}) = \varepsilon$.

We  can define a topology of weak convergence in the space $\Qc_p^\lambda(\Xc,\Yc)$.
\begin{definition}
\label{d:kernel-convergence}
The sequence of kernels $\{Q_k\}$ converges weakly to $Q$ in $\Qc_p^\lambda(\Xc,\Yc)$,
{  where $\lambda \in \Pc_p(\Xc)$}, if for every {  continuous}  function $f:\Xc \times \Yc \to \Rb$ such that
${  |f(x,y)|} \le 1 + d\left(y_{0}, y\right)^{p}$, $\forall\,(x\in \Xc,\,y\in \Yc)$,
\[
\int_\Xc \int_\Yc f(x,y)\,  Q_k({\D y}|x)\; \lambda({\D x}) \longrightarrow  \int_\Xc \int_\Yc f(x,y)\,  Q({\D y}|x)\; \lambda({\D x}).
\]
\end{definition}
This entails that $\lambda \circledast Q_k  \rightharpoonup \lambda \circledast Q$, and, due to Definition \ref{d:villani-def-1}(i),
$\lambda \circ Q^k \overset{p}\to  \lambda \circ Q$.
{  The latter property is essential to our approximation scheme, because it allows us to derive the convergence of integrals or
other functionals of the mixture distributions in the space $\Pc_p(\Xc)$.}
 It also implies that
$\Wc_p^\lambda(Q_k,\delta_{\{y_0\}}) \to \Wc_p^\lambda(Q,\delta_{\{y_0\}})$ (see  \eqref{dist-to-delta} in the Appendix).

The  distance $\Wc_p^\lambda(\cdot,\cdot)$ metrizes the topology of weak convergence in $\Qc_p^\lambda(\Xc,\Yc)$.
\begin{theorem}
\label{t:kernel-distance-metric}
Let $\Xc$ and $\Yc$ be Polish spaces, $p\in [1,\infty)$, and $\lambda\in \Pc_P(\Xc)$. Then the following
statements are equivalent: {\rm (i)}
$Q_k \to Q$ weakly in $\Qc_p^\lambda(\Xc,\Yc)$; {\rm (ii)}
$\Wc_p^\lambda(Q_k ,Q) \to 0$.
\end{theorem}
{  The proof is provided in the Appendix.}

By the triangle inequality, we obtain the following corollary.
\begin{corollary}
\label{c:continuity-kernel-dist}
The functional $\Wc_p^\lambda(\cdot,\cdot)$ is continuous with respect to the weak convergence in the space
$\Qc_p^\lambda(\Xc,\Yc) \times \Qc_p^\lambda(\Xc,\Yc)$.
\end{corollary}

{
We can also establish an extension of the Kantorovich--Rubinstein duality.
\begin{theorem}
    \label{t:KR-kernels}
    For all $Q,\widetilde{Q}\in \Qc_1^\lambda(\Xc,\Yc)$ we have
    \begin{equation}
    \label{KR-kernels}
\begin{aligned}
\Wc_1^\lambda(Q,\widetilde{Q}) &= \sup_{f(\cdot,\cdot)\in F}\left\{\int_{\mathcal{X}\times \mathcal{Y}} f(x,y)\; (\lambda \circledast Q)(\D x\,\D y)-\int_{\mathcal{X}\times \mathcal{Y}} f(x,y) (\lambda \circledast \widetilde{Q})(\D x\,\D y)\right\},
\end{aligned}
\end{equation}
where $F$ is the set of measurable functions on $\Xc \times \Yc$ such that $\|f(x,\cdot)\|_{\text{\rm Lip}} \leq 1$ for $\lambda$-almost all  $x\in \Xc$.
With no loss of generality, we may also assume that $f(\cdot,y_0)\equiv 0$, for all $f\in F$.
\end{theorem}
The proof is provided in the appendix.

}

\section{Approximate Risk Evaluation in Markov Systems}
\label{s:abstract-method}

Our objective in this section is to propose and analyze a method for approximating  forward--backward
Markov systems which are described by \eqref{kernel-model}-\eqref{DP-risk-finite},
with the use of the the integrated transportation distance as the criterion for constructing the approximation and a measure of its accuracy.
Throughout this section, the parameter $p\in [1,\infty)$ is fixed.

The method proceeds in stages, for $t=0,1,\dots,T$. At each stage $t$, for all $\tau=0,\dots,t-1$,
we already have approximate transition kernels $\widetilde{Q}_\tau:\Xc \to\Pc_p(\Xc)$, $\tau=0,\dots,t-1$. These kernels define the approximate marginal distribution
\begin{equation}
\label{tilde-lambda}
\widetilde{\lambda}_t = \lambda_0\circ \widetilde{Q}_0 \circ \widetilde{Q}_1 \circ \dots \circ \widetilde{Q}_{t-1} = \widetilde{\lambda}_{t-1}\circ \widetilde{Q}_{t-1}.
\end{equation}
We also have the subspaces $\Xc_\tau\subset \Xc$ as $\Xc_\tau = \text{supp}(\widetilde{\lambda}_\tau)$, $\tau=0,1,\dots,t$.
For $t=0$, $\widetilde{\lambda}_0=\lambda_0$, and
$\Xc_0=\supp{\lambda_0}$.

At the stage $t$, we construct a kernel $\widetilde{Q}_t:\Xc_t \to \Pc_p(\Xc)$ such that
\begin{equation}
\label{Delta}
\Wc_p^{\widetilde{\lambda}_t}(Q_t,\widetilde{Q}_t)\le \varDelta_t.
\end{equation}
If $t< T-1$, we increase $t$ by one, and continue; otherwise, we stop.
Observe that the approximate marginal distribution $\widetilde{\lambda}_t$ is well-defined at each step of this abstract scheme.

We then solve the approximate version of the risk evaluation algorithm \eqref{DP-risk-finite}, with the true kernels $Q_t$
replaced by the approximate kernels $\widetilde{Q}_t$, $t=0,\dots,T-1$:
\begin{equation}
\label{DP-risk-approximate}
\widetilde{v}_t(x) = c_t(x) +  \sigma_t\big(x,\widetilde{Q}_t(x), \widetilde{v}_{t+1}(\cdot)\big),\quad x\in \Xc_t, \quad t=0,1,\dots,T-1;
\end{equation}
we assume that $\widetilde{v}_T(\cdot)\equiv v_T(\cdot)\equiv c_T(\cdot)$.

Our plan is to estimate the error of this evaluation in terms of the kernel errors $\varDelta_t$.
For this purpose, we make the following general assumptions.
\begin{description}
\item[(A1)] For every $t=0,1,\dots,T-1$ and for every $x\in \Xc_t$, the operator $\sigma_t(x,\,\cdot\, , v_{t+1})$ is Lipschitz continuous
with respect to the metric $W_p(\cdot,\cdot)$ with the constant $L_t$:
\[
 \big| \sigma_{t}\big(x,\mu, {v}_{t+1}(\cdot)\big) - \sigma_{t}\big(x,\nu, {v}_{t+1}(\cdot)\big)\big|\\
\le L_{t}\, W_p(\mu,\nu), \quad \forall\, \mu,\nu \in \Pc_p(\Xc);
\]
\item[(A2)] For every $x\in \Xc_t$ and for every $t=0,1,\dots,T-1$, the operator $\sigma_t(x,\widetilde{Q}_t(x), \,\cdot\,)$ is
Lipschitz continuous
with respect to the norm in the space $\Lc_p(\Xc,\Bc(\Xc),\widetilde{Q}_t(x))$ with the constant $K_t$:
\begin{multline*}
 \big| \sigma_{t}\big(x,\widetilde{Q}_t(x), {v}(\cdot)\big) - \sigma_{t}\big(x,\widetilde{Q}_t(x), {w}(\cdot)\big)\big|
\le K_{t}\, \|v - w\|_p,\\  \forall\,v,w\in \Lc_p(\Xc,\Bc(\Xc),\widetilde{Q}_t(x)).
\end{multline*}
\end{description}
These are fairly schematic conditions, but they are exactly what we need for the analysis below. After the theorem, we discuss
several important cases, in which these conditions are satisfied.
\begin{theorem}
\label{t:error_estimate}
If assumptions (A1)--(A2) are satisfied, then for all  $t=0,\dots,T-1$ we have
\begin{equation}
\label{error-p}
\bigg(\int_\Xc \big| \widetilde{v}_{t}(x) - v_{t}(x) \big|^p\;\widetilde{\lambda}_t({\D x})\bigg)^{1/p} \le \sum_{\tau=t}^{T-1}L_\tau\bigg(\prod_{j=t}^{\tau-1}K_j\bigg)
\varDelta_\tau.
\end{equation}
In particular, for {  $t=0$},
\begin{equation}
\label{error_estimate-1}
\big| \widetilde{v}_{0}(x_0) - v_{0}(x_0) \big| \le
\sum_{\tau=0}^{T-1}L_\tau\bigg(\prod_{j=0}^{\tau-1} K_j\bigg)
\varDelta_\tau.
\end{equation}
\end{theorem}
\begin{proof}
First, we prove by induction backward in time that for all  $t=0,1,\dots,T-1$ and all $x\in \Xc_t$  we have
\begin{equation}
\label{error_estimate}
\big| \widetilde{v}_{t}(x) - v_{t}(x) \big| \le
\sum_{\tau=t}^{T-1}L_\tau\bigg(\prod_{j=t}^{\tau-1}K_j\bigg)
\Wc_p^{\delta_x \circ \widetilde{Q}_{t}\circ\dots \circ \widetilde{Q}_{\tau-1}}(\widetilde{Q}_{\tau},{Q}_{\tau}).
\end{equation}
At the time $t=T-1$, assumption (A1) yields the inequality
\begin{multline*}
\big| \widetilde{v}_{T-1}(x) - v_{T-1}(x) \big| \le
\Big| \sigma_{T-1}\big(x,\widetilde{Q}_{T-1}(x), {v}_{T}(\cdot)\big) - \sigma_{T-1}\big(x,{Q}_{T-1}(x), {v}_{T}(\cdot)\big)\Big|\\
\le L_{T-1}\, W_p(\widetilde{Q}_{T-1}(x),{Q}_{T-1}(x)) =  L_{T-1}\, \Wc_p^{\delta_x}(\widetilde{Q}_{T-1},{Q}_{T-1}),
\end{multline*}
which is the same as \eqref{error_estimate} for $T-1$.
Supposing \eqref{error_estimate} is true for $t$, we verify it for $t-1$. Using assumptions (A1) and (A2) we obtain:
\begin{align*}
\lefteqn{\big| \widetilde{v}_{t-1}(x) - v_{t-1}(x) \big|}\\
 &\le
\Big| \sigma_{t-1}\big(x,\widetilde{Q}_{t-1}(x), {v}_{t}(\cdot)\big) - \sigma_{t-1}\big(x,{Q}_{t-1}(x), {v}_{t}(\cdot)\big)\Big|\\
& {\quad}+ \Big| \sigma_{t-1}\big(x,\widetilde{Q}_{t-1}(x), \widetilde{v}_{t}(\cdot)\big) - \sigma_{t-1}\big(x,\widetilde{Q}_{t-1}(x), {v}_{t}(\cdot)\big)\Big|\\
&\le L_{t-1}\, W_p(\widetilde{Q}_{t-1}(x),{Q}_{t-1}(x))
 + K_{t-1} \bigg(\int_\Xc \big|\widetilde{v}_{t}(y)-{v}_{t}(y)\big|^p\; \widetilde{Q}_{t-1}(\D y| x)\bigg)^{1/p}.
\end{align*}
The substitution of \eqref{error_estimate} and the application of the Minkowski inequality yield
\begin{multline*}
\big| \widetilde{v}_{t-1}(x) - v_{t-1}(x) \big|
\le L_{t-1}\, \Wc_p^{\delta_x}(\widetilde{Q}_{t-1},{Q}_{t-1})\\
 + K_{t-1} \sum_{\tau=t}^{T-1}L_\tau\bigg(\prod_{j=t}^{\tau-1}K_j\bigg)
\bigg(\int_\Xc \big[\Wc_p^{\delta_y \circ \widetilde{Q}_{t}\circ\dots \circ \widetilde{Q}_{\tau-1}}(\widetilde{Q}_{\tau},{Q}_{\tau})\big]^p\;
\widetilde{Q}_{t-1}(\D y|x)\bigg)^{1/p}.
\end{multline*}
Observing that
\begin{equation}
    \label{key-identity-1}
\int_\Xc \big[\Wc_p^{\delta_y \circ \widetilde{Q}_{t}\circ\dots \circ \widetilde{Q}_{\tau-1}}(\widetilde{Q}_{\tau},{Q}_{\tau})\big]^p\;
\widetilde{Q}_{t-1}(\D y|x) = \big[\Wc_p^{\delta_x \circ \widetilde{Q}_{t-1}\circ\widetilde{Q}_t \circ\dots \circ \widetilde{Q}_{\tau-1}}(\widetilde{Q}_{\tau},{Q}_{\tau})\big]^p,
\end{equation}
we can write the preceding displayed inequality as
\begin{multline*}
\big| \widetilde{v}_{t-1}(x) - v_{t-1}(x) \big|
\le L_{t-1}\, \Wc_p^{\delta_x}(\widetilde{Q}_{t-1},{Q}_{t-1})\\
 + K_{t-1} \sum_{\tau=t}^{T-1}L_\tau\bigg(\prod_{j=t}^{\tau-1}K_j\bigg) \Wc_p^{\delta_x \circ \widetilde{Q}_{t-1}\circ\widetilde{Q}_t \circ\dots \circ \widetilde{Q}_{\tau-1}}(\widetilde{Q}_{\tau},{Q}_{\tau}),
\end{multline*}
which is the same as \eqref{error_estimate} for $t-1$. By induction, \eqref{error_estimate} is true for all $t$.

The formula \eqref{error-p} follows now by integrating the right-hand side of \eqref{error_estimate} and using the identity
\begin{equation}
    \label{key-identity-2}
\int_\Xc \big[\Wc_p^{\delta_x \circ \widetilde{Q}_{t}\circ\dots \circ \widetilde{Q}_{\tau-1}}(\widetilde{Q}_{\tau},{Q}_{\tau})\big]^p\;
\widetilde{\lambda}_t({\D x}) = \big[\Wc_p^{\widetilde{\lambda}_{\tau}}(\widetilde{Q}_{\tau},{Q}_{\tau})\big]^p, \quad \tau=t,\dots,T-1.
\end{equation}
The formula \eqref{error_estimate-1} is a special case of \eqref{error-p} {  resulting from
$\lambda_0=\widetilde{\lambda}_0=\delta_{x_0}$}.
\end{proof}
{
\begin{remark}
\label{r:interpretation}
{\rm At each time $t$, the ingredients of the formula \eqref{error-p}: $\widetilde{\lambda}_t$ and $\varDelta_t$,
are known. The identities  \eqref{key-identity-1} and \eqref{key-identity-2} explain  the use of the marginal $\tilde{\lambda}$ in \eqref{Delta},
and the mechanism of the error control. Compared to  \cite[Thm. 6.2]{zahle2022concept}, which deals with the expected value problem
in the backward system,  the error estimate \eqref{error_estimate-1} is linear in the $\varDelta_\tau$'s, $\tau= 1,\dots,T-1$.}
\end{remark}
}

Assumptions (A1) and (A2) can be verified in several relevant special cases.

\begin{example}
\label{e:nested}
{\rm
Consider the transition risk mappings of the following form:
\begin{multline}
\label{composite}
\sigma(x,\mu,v) = \\
\Eb_\mu\Big[f_1\Big(x,\Eb_\mu\big[f_2\big(x,\Eb_\mu[\;\cdots f_k(x,\Eb_\mu[f_{k+1}(x,v(\cdot))],v(\cdot))]\;\cdots,v(\cdot)\big)\big],v(\cdot)\Big)\Big],
\end{multline}
where $v:\Xc \to \Rb$, $\Eb_\mu\big[f(v(\cdot))\big] = \int_\Xc f(v(y))\;\mu({\D y})$, and $f_j:\Xc\times\Rb^{m_j} \times \Rb \to \Rb^{m_{j-1}}$, $j=1,\dots,k$, with $m_0=1$ and $f_{k+1}:\Xc\times\Rb\to\Rb^{m_k}$.  This is a fairly general class, considered in \cite{dentcheva2017statistical}, which covers several risk measures, such as the mean--semideviation measure \eqref{msd-form}.
Indeed, if $p=1$, we can write \eqref{msd-form} in the form \eqref{composite}, with $k=1$ and
$f_1(x,\eta,v(\cdot)) = \eta + \varkappa(x)[ v(\cdot) - \eta]_+$,
$f_2(x,v(\cdot)) = v(\cdot)$.

The model \eqref{composite} also covers the mapping \eqref{sigma-stopping} in the stopping problem.
In this case, $k=1$ again, and
$f_1(x,\eta,v(\cdot)) = \max(r(x); \eta)$,
$f_2(x,v(\cdot)) = v(\cdot)$.

Suppose the functions $f_j(x,\cdot,\cdot)$ , $j=1,\dots,k$, and  $f_{k+1}(x,\cdot)$  are Lipschitz continuous (it is true in both special cases mentioned above). Furthermore, let the function $v(\cdot)$ be Lipschitz continuous as well. Then,
by virtue of the Kantorovich--Rubinstein duality, the functional
$ g_{k+1}(\mu) =  \Eb_\mu\big[f_{k+1}(x,v(\cdot))\big]$
is Lipschitz continuous in the space $\Pc_1(\Xc)$. In a similar way, the mapping
$g_{k}(\mu)= \Eb_\mu\big[f_k(x,\Eb_\mu[f_{k+1}(x,v(\cdot))],v(\cdot))\big] = \Eb_\mu\big[f_k(x,g_{k+1}(\mu),v(\cdot))\big]$
is Lipschitz continuous as well. Proceeding in this way, we conclude that assumption (A1) is satisfied with $p=1$, as long as the
optimal value functions $v_t(\cdot)$, $t=1,\dots,T$, are Lipschitz continuous.

Consider assumption (A2). With a fixed measure ${\mu}$ (corresponding to $\widetilde{Q}_{t}(x)$ in (A2)), we observe that
the functional
$\varphi_{k+1}(v) = \Eb_\mu\big[f_{k+1}(x,v(\cdot))\big]$
is Lipschitz continuous in the space $\Lc_1(\Xc,\Bc(\Xc),\mu)$. This, in turn, implies that the functional
\[
\varphi_{k}(v)= \Eb_\mu\big[f_k(x,\Eb_\mu[f_{k+1}(x,v(\cdot))],v(\cdot))\big] = \Eb_\mu\big[f_k(x,\varphi_{k+1}(v),v(\cdot))\big]
\]
is Lipschitz continuous in $\Lc_1(\Xc,\Bc(\Xc),\mu)$. Proceeding in a similar way, we conclude that the assumption (A2) is satisfied as well.
\hfill$\Box$
}
\end{example}
\begin{example}
\label{e:avar}
{\rm
Consider the transition risk mapping \eqref{avar-def} derived from the Average Value at Risk.
If $v(\cdot)$ is Lipschitz continuous, then the mapping
$\mu \mapsto \avar_\alpha(x,\mu,v(\cdot))$
 is Lipschitz continuous on the space $\Pc_1(\Xc)$. Thus, assumption (A1) is satisfied with $p=1$. Furthermore,
for a fixed $\mu$, the mapping $v(\cdot) \mapsto \Eb_{\mu}\big[\max(0,v(y)-\eta)\big]$ is Lipschitz continuous (with the modulus 1)
in the space $\Lc_1(\Xc,\Bc(\Xc),\mu)$. Indeed, suppose $\eta_v$ achieves the infimum in \eqref{avar-def}. Then
\begin{multline*}
\avar_\alpha(x,\mu,w(\cdot)) - \avar_\alpha(x,\mu,v(\cdot)) \\
\le \frac{1}{\alpha} \Eb_{\mu}\big[\max(0,w(y)-\eta_v)\big] - \frac{1}{\alpha} \Eb_{\mu}\big[\max(0,v(y)-\eta_v)\big]
\le \frac{1}{\alpha} \|w -v \|_1.
\end{multline*}
Reversing the roles of $v$ and $w$ we obtain the Lipschitz continuity of \eqref{avar-def} on the space $\Lc_1(\Xc,\Bc(\Xc),\mu)$. If the infimum
is not achieved, which may happen for $\alpha=1$, then $\avar_1(x,\mu,v(\cdot))= \Eb_\mu[v(\cdot)]$ and the Lipschitz continuity is evident.
Therefore, assumption (A2) is satisfied. \hfill $\Box$
}
\end{example}
The last example allows for deriving the Lipschitz continuity in the space $\Pc_1(\Xc)$ of a broad class
of coherent risk mappings in the spectral form  \eqref{spectral}, or, more generally, enjoying the Kusuoka representation \cite{kus:01}.
We refer the reader to \cite[Thm. 6.5]{batch} for the details.
\begin{example}
\label{msd-p}
{\rm
Consider now the mean--semideviation mapping \eqref{msd-form} for $p>1$. By \cite[Lem. 6.6]{batch}, if $v(\cdot)$ is Lipschitz continuous, then the functional $\mu \mapsto \msd_p(x,\mu,v)$ is Lipschitz continuous on the space $\Pc_p(\Xc)$. Thus, assumption (A1) is satisfied.

Furthermore, for a fixed $\mu$, the continuity of the mapping $v \mapsto \msd_p(x,\mu,v)$ on the space $\Lc_p(\Xc,\Bc(\Xc),\mu)$
is evident, because it is a sum of a linear mapping and the norm. Thus, (A2) holds true as well.
\hfill $\Box$
}
\end{example}
It follows from the above examples that the assumptions (A1) and (A2) are indeed satisfied for a wide range of transition risk mappings.
The Lipschitz continuity of the value functions $v_t(\cdot)$, $t=2,\dots,T$, is crucial in this context.

This can be guaranteed by a simple induction argument. Suppose
each function $c_t(\cdot)$  and operator
$(x,\mu) \mapsto \sigma_t(x,\mu,v_{t+1}(\cdot))$
are Lipschitz continuous in $\Xc$ and $\Xc \times \Qc^{\lambda_t}(\Xc,\Xc)$, respectively, provided the function $v_{t+1}(\cdot)$
is Lipschitz continuous?  Moreover, let the kernels $Q_t:\Xc \to \Pc_p(\Xc)$, $t=1,\dots,T-1$, be Lipschitz continuous as well:
a constant $L_Q$ exists, such that
\begin{equation}
    \label{Dobrushin}
W_p(Q_t(x),Q_t(x')) \le L_{Q_t}\, d(x,x'),\quad \forall\;x,x'\in \Xc.
\end{equation}
 Then the function $v_t(\cdot)$ in \eqref{DP-risk-finite} is a composition
of Lipschitz continuous mappings, and it thus Lipschitz continuous. By induction, all value functions are Lipschitz continuous.
{  Their Lipschitz constants, though, may grow exponentially with the horizon $T-t$, if  $L_Q>1$. The constant
$L_Q$ is known as the \emph{ergodicity coefficient}; see \cite{rudolf2018perturbation} and the references therein.
}

We can also study the accuracy of the marginal distributions $\widetilde{\lambda}_t$, $t=1,\dots,T$. First, we
establish a useful continuity result.

\begin{lemma}
\label{f:muQ-cont}
If a kernel $Q:\Xc \to \Pc_p(\Xc)$ is Lipschitz continuous, then the mapping $\mu \mapsto \mu \circ Q$ is Lipschitz continuous on
$\Pc_p(\Xc)$ with the same modulus.
\end{lemma}
\begin{proof}
If $\lambda({\D y}\,{\D y}'| x,x')$ is the optimal transport plan from $Q(x)$ to $Q(x')$, then
\[
W_p(Q(x), Q(x'))^p = \int_{\Xc\times\Xc} d(y,y')^p \;\lambda({\D y}\,{\D y}'| x,x') \le L_Q^p d(x,x')^p,
\]
where $L_Q$ is the Lipschitz constant of $Q$.
Suppose $\pi({\D x}\,{\D x}')$ is the optimal coupling of $\mu$ and $\nu$.
Consider the transport plan $\varPi = \pi \circ \lambda$, with $\pi$ considered as a marginal on $\Xc \times \Xc$, and
$\lambda$ as a kernel from $\Xc \times \Xc$ to $\Pc(\Xc \times \Xc)$. We have
\begin{align*}
\varPi(A \times \Xc) &= \int_{\Xc\times\Xc} \int_{\Xc}\lambda (A,{\D y}'|x,x')\;\pi({\D x}\,{\D x}')\\
&= \int_{\Xc\times\Xc} Q(A|x) \;\pi({\D x}\,{\D x}') = \int_{\Xc} Q(A|x) \;\mu({\D x}')= [\mu\circ Q](A).
\end{align*}
In a similar way,
$\varPi(\Xc \times B)  [\nu\circ Q](B)$, and thus $\varPi$ is a feasible transport plan from $\mu\circ Q$ to $\nu\circ Q$.
Therefore
\begin{align*}
W_p\big(\mu\circ Q,\nu\circ Q\big) &\le \int_{\Xc \times\Xc}d(y,y')^p\;\varPi({\D y}\,{\D y}')\\
&=  \int_{\Xc \times\Xc}\int_{\Xc \times\Xc} d(y,y')^p\;\lambda({\D y}\,{\D y}'| x,x')\;\pi({\D x}\,{\D x}')\\
&\le L_Q^p\int_{\Xc \times\Xc}d(x,x')^p\;\pi({\D x}\,{\D x}') = L_Q^p W_p(\mu,\nu)^p.
\end{align*}
It follows that $L_Q$ is the Lipschitz constant of the mapping $\mu \mapsto \mu \circ Q$.
\end{proof}
We can now easily estimate the errors of the marginal distributions.
\begin{theorem}
\label{t:lambda-error}
If the kernels $Q_t:\Xc\to \Pc_p(\Xc)$ are Lipschitz continuous with constants $L_{Q_t}$, then
\begin{equation}
\label{lambda-error}
W_p(\widetilde{\lambda}_t , \lambda_t) = \sum_{\tau=1}^{t-1}\varDelta_\tau \prod_{i=\tau+1}^{t-1}L_{Q_i}, \quad t=1,\dots,T.
\end{equation}
\end{theorem}
\begin{proof}
The estimate \eqref{lambda-error} is true for $t=1$. Supposing it is valid for $t-1$, we verify it for~$t$:
\begin{align*}
W_p(\widetilde{\lambda}_t , \lambda_t) &=W_p(\widetilde{\lambda}_{t-1}\circ \widetilde{Q}_{t-1},  \lambda_{t-1}\circ Q_{t-1})\\
&\le W_p(\widetilde{\lambda}_{t-1}\circ \widetilde{Q}_{t-1},  \widetilde{\lambda}_{t-1}\circ {Q}_{t-1})
+ W_p( \widetilde{\lambda}_{t-1}\circ {Q}_{t-1},\lambda_{t-1}\circ Q_{t-1}) \\
& \le {\Wc}_p^{\widetilde{\lambda}_{t-1}}\big(\widetilde{Q}_{t-1},  {Q}_{t-1}) + W_p( \widetilde{\lambda}_{t-1}\circ {Q}_{t-1},\lambda_{t-1}\circ Q_{t-1})\\
&\le \varDelta_{t-1} + L_{Q_{t-1}} W_p( \widetilde{\lambda}_{t-1},\lambda_{t-1}).
\end{align*}
The substitution of \eqref{lambda-error} for $t-1$ yields the same estimate for $t$. By induction, it is true for all $t$.
\end{proof}

\section{Kernel Approximation by Particles}
\label{s:particles}

In the general method discussed in the previous section, we iteratively constructed approximate kernels $\widetilde{Q}_t$, proceeding from
$t=0$ to $t=T-1$, and we used their error estimates \eqref{Delta} to estimate the error of the risk evaluation.

Now we aim at an implementable method to realize this general scheme. The most important assumption is that the spaces $\Xc_t$, $t=0,1,\dots,T$,
 be \emph{finite}. We assume that we start from  $\Xc_0=\{x_0\}$.
At each stage $t$, we aim to construct a finite set $\Xc_{t+1} \subset \Xc$ of cardinality $M_{t+1}$ and a kernel $\widetilde{Q}_t:\Xc_t \to \Pc(\Xc_{t+1})$
by solving the following  problem:
\begin{equation}
\label{kernel-search}
\min_{\Xc_{t+1},\widetilde{Q}_t} \  \Wc_p^{\widetilde{\lambda}_t}(Q_t,\widetilde{Q}_t)\quad
\text{s.t.} \quad \supp(\widetilde{\lambda}_t \circ \widetilde{Q}_t) = \Xc_{t+1}\quad \text{and}\quad \big| \Xc_{t+1} \big| \le M_{t+1}.
\end{equation}
After (approximately) solving this problem, we increase $t$ by one and continue. Evidently, the objective function of this problem
is motivated by its direct effect on the error estimates in Theorems \ref{t:error_estimate} and \ref{t:lambda-error}.

Let us focus on effective and scalable ways for constructing an approximate solution to problem \eqref{kernel-search}.
We represent the (unknown) support of $\widetilde{\lambda}_t \circ \widetilde{Q}_t$ by $\Xc_{t+1} = \big\{ z_{t+1}^j\big\}_{j=1,\dots,M_{t+1}}$
and the (unknown) transition probabilities by $\widetilde{Q}(z_{t+1}^j|z_t^s)$, $s=1,\dots,M_n$, $j=1,\dots,M_{n+1}$.
With the use of Definition \ref{d:kernel-distance},   problem \eqref{kernel-search} can be equivalently rewritten as:
\begin{equation}
\label{kernel-search-2}
\begin{aligned}
\min_{\Xc_{t+1},\widetilde{Q}_t} & \; \sum_{s=1}^{M_n} \widetilde{\lambda}_t^s W_p\big(Q_t(\cdot|z_t^s),\widetilde{Q}_t(\cdot|z_t^s)\big)^p\\
\text{s.t.}&\; \supp \big(\widetilde{Q}_t(\cdot|z_t^s)\big) \subset \Xc_{t+1},\quad s=1,\dots,M_n,\\
&\ \big| \Xc_{t+1} \big| \le M_{t+1}.
\end{aligned}
\end{equation}
Let $\pi_t^s$ be a transportation plan from $Q_t(\cdot|z_t^s)$ to $\widetilde{Q}_t(\cdot|z_t^s)$. Then it follows from the definition of the
Wasserstein distance that $W_p\big(Q_t(\cdot|z_t^s),\widetilde{Q}_t(\cdot|z_t^s)\big)^p$ is the optimal value of the problem
\begin{equation}
\label{kernel-search-3}
\begin{aligned}
\min_{\pi_t^s \ge 0} &\; \sum_{j=1}^{M_{t+1}}\int_\Xc\| x - z_{t+1}^j\|^p\; \pi_t^{sj}({\D x})\\
\text{s.t.} &\; \int_{\Xc}\pi_t^{sj}({\D x}) = \widetilde{Q}_t(z_{t+1}^j|z_t^s),\quad j=1,\dots,M_{t+1},\\
&\; \sum_{j=1}^{M_{t+1}}\pi_t^{sj}(A) = Q_t(A|z_t^s),\quad \forall\,A\in \Bc(\Xc).
\end{aligned}
\end{equation}
The integration of problems \eqref{kernel-search-2}--\eqref{kernel-search-3} leads to a very difficult nonconvex infinite-dimension\-al problem
which can be only solved in very special cases. To develop a tractable approach in large-scale applications, we restrict the supports
of the kernels under consideration to   finite sets.
{  We may remark that the optimal quantization of probability distributions with the use of the Wasserstein metric was systematically studied in \cite{graf2007foundations}. Our problem is slightly different because we want to obtain a ``quantization'' of kernels.
}

In our particle approach, for $t=0,1,\dots,T-1$, each distribution $Q_t(\cdot|z_t^s)$ is represented by finitely many points (particles) $\big\{x_{t+1}^{s,i}\big\}_{i\in \Ic_{t+1}^s}$,  drawn independently from $Q_t(\cdot|z_t^s)$. If the state space $\Xc$ is finite-dimensional,
the expected error of this approximation is well-investi\-ga\-ted in \cite{dereich2013constructive,fournier2015rate}, as a function of the sample size $\big| \Ic_{t+1}^s\big|$, the dimension of the state space, and the moments of the distribution
{ (see formula \eqref{fournier} below)}. From this point, we consider the error of this large-size
discrete approximation as fixed, and we focus on constructing smaller support with as small an error to
the particle distributions as possible.
To this end, we introduce the sets
$\Zc_{t+1} =\big\{\zeta^k_{t+1}\big\}_{k=1,\dots,K_{t+1}}$,
which are pre-selected potential locations of the next-stage representative states
$z_{t+1}^j$, $j=1,\dots,M_{t+1}$. In the simplest case, we may
consider the union of the sets of particles, $\big\{ x_{t+1}^{s,i}, \ i\in \Ic_{t+1}^s,\ s=1,\dots,M_t\big\}$ as the potential locations,
but often computational efficiency requires that $K_{t+1} \ll \sum_{s=1}^{M_t} \big|\Ic_{t+1}^s\big|$.
There are several heuristic ways to choose the set $\Zc_{t+1}$ of potential points. For instance, they may be sampled independently along with successors at the particle generation step, or they may be sampled from a different distribution.
In any case, we still have $M_{t+1} \ll K_{t+1}$, which makes the problem of finding the best representative points nontrivial.

Suppose temporarily the next-stage representative points $\big\{ z_{t+1}^j\big\}_{j=1,\dots,M_{t+1}}$ have been found. Then the particle version of problem
\eqref{kernel-search-3} (for a fixed $s$) takes on the form:
\begin{equation}
\label{kernel-search-4}
\begin{aligned}
\min_{\pi_t^s \ge 0} &\; \sum_{j=1}^{M_{t+1}}\sum_{i\in \Ic_{t+1}^s}\| x_{t+1}^{s,i} - z_{t+1}^j\|^p\; \pi_t^{s,i,j}\\
\text{s.t.} 
&\; \sum_{j=1}^{M_{t+1}}\pi_t^{s,i,j} = \frac{1}{|\Ic_{t+1}^s|},\quad i\in \Ic_{t+1}^s.
\end{aligned}
\end{equation}
It has a straightforward solution: find for each particle $i$ the closest representative point,
$
j^*(i)= \argmin_{j=1,\dots,M_{t+1}} \| x_{t+1}^{s,i} - z_{t+1}^j\|,
$
and set
$\pi_t^{s,i,j^*(k)}  = \frac{1}{|\Ic_{t+1}^s|}$;
for other $j$, we set it to 0. The implied approximate kernel is
\begin{equation}
\label{implied-kernel}
\widetilde{Q}_t(z_{t+1}^j|z_t^s) = \sum_{i\in \Ic_{t+1}^s} \pi_t^{s,i,j},\quad s=1,\dots,M_t,\quad j=1,\dots,M_{t+1},
\end{equation}
which simply counts the particles from $\Ic_{t+1}^s$ which were assigned to $z_{t+1}^j$.

These considerations allow us to integrate problems \eqref{kernel-search-4} into \eqref{kernel-search-2}. We introduce the binary variables
\[
\gamma_k = \begin{cases} 1 & \text{ if the point $\zeta_{t+1}^k$ has been selected to $\Xc_{t+1}$},\\
0 & \text{ otherwise},
\end{cases}
\quad k=1,\dots,K_{t+1},
\]
and we re-scale the transportation plans:
\[
\beta^{s,i,k} = |\Ic_{t+1}^s|\pi_t^{s,i,k}, \quad s=1,\dots,M_t,\  i \in \Ic_{t+1}^s,\  k=1,\dots,K_{t+1}.
\]
We obtain from \eqref{kernel-search-2} the following linear mixed-integer optimization problem:
\begin{equation}
\label{mixed-bin}
\begin{aligned}
\min_{\gamma,\beta}
&\; \sum_{s=1}^{M_n}  \frac{\widetilde{\lambda}_t^s}{|\Ic_{t+1}^s|}\sum_{k=1}^{K_{t+1}}\sum_{i\in \Ic_{t+1}^s}\| x_{t+1}^{s,i} - \zeta_{t+1}^k\|^p\; \beta^{s,i,k}\\
\text{s.t.} &\;\beta^{s,i,k} \le \gamma_k,\quad s=1,\dots,M_t,\quad i \in \Ic_{t+1}^s,\quad k=1,\dots,K_{t+1}, \\
&\; \sum_{k=1}^{K_{t+1}} \beta^{s,i,k} = 1, \quad s=1,\dots,M_t,\quad i \in \Ic_{t+1}^s,\\
&\; \sum_{k=1}^{K_{t+1}} \gamma_k \le M_{t+1},\\
&\; \beta^{s,i,k} \in[0,1],\  \gamma_k\in \{0,1\}, \quad s=1,\dots,M_t,\  i \in \Ic_{t+1}^s,\  k=1,\dots,K_{t+1}.
\end{aligned}
\end{equation}
The complicating element is that the $\gamma_k$'s are binary variables. However, we may solve the relaxation of \eqref{mixed-bin}
in which we require only that $\gamma_k \in [0,1]$, $k=1,\dots,K_{t+1}$, { while still bounding their sum by $M_{t+1}$}. After that, we may randomly assign to fractional
$\gamma_k$'s values 0 or 1, by using independent Bernoulli random variables with parameters $\gamma_k$, and then resolve
\eqref{mixed-bin} with respect to the $\beta$ variables only. This can be accomplished by assigning each point $x_{t+1}^{s,i}$ to the
closest $\zeta_{t+1}^k$ having $\gamma_k=1$. The implied approximate kernel is given by \eqref{implied-kernel}:
\begin{equation}
\label{implied-kernel-2}
\widetilde{Q}_t(\zeta_{t+1}^k|z_t^s) = \frac{1}{|\Ic_{t+1}^s|}\sum_{i\in \Ic_{t+1}^s} \beta^{s,i,k},\quad s=1,\dots,M_t,\quad k=1,\dots,K_{t+1}.
\end{equation}
By construction, these probabilities can be positive only when $\gamma_k=1$.

Finally, $\widetilde{\lambda}_{t+1} = \widetilde{\lambda}_t \circ \widetilde{Q}_t$, and the iteration continues until $t=T$.

At each stage $t$, the estimate of the error $\varDelta_t$ in \eqref{Delta} can be computed: it is the sum of the $p$-th root
of the objective value of \eqref{mixed-bin}
 and the particle distribution error.
 {
 Denoting by $\widehat{Q}_t$ the approximate kernel defined by all the particles sampled, due to Theorem~\ref{TS_T}, we have
 \[
\Wc_p^{\tilde{\lambda}_t}(\widetilde{Q}_t,Q_t) \le
\Wc_p^{\tilde{\lambda}_t}(\widetilde{Q}_t,\widehat{Q}_t)  + \Wc_p^{\tilde{\lambda}_t}(\widehat{Q}_t,Q_t).
 \]
 To recall a bound on the expected value of the second term, we assume that the state space is
 finite-dimensional, and that for each point $z_t^s$ the measure $Q_t(\cdot|z_t^s)$ has
 a finite moment $m_u$ for some $u>p$.
 The following
inequality due to \cite{dereich2013constructive,fournier2015rate} is true for all $N=|\Ic_{t+1}^s|$:
\begin{multline}
\label{fournier}
\Eb \big[ W_p\big(\widehat{Q}_t(\cdot|z_t^s),Q_t(\cdot|z_t^s\big)\big] \le C m_{u}^{p/{u}} \times \\
\times \begin{cases}
N^{-1/2} + N^{-({u}-p)/{u}} & \text{ if } p>n/2 \text{ and } {u}\ne 2p,\\
N^{-1/2}\ln(1+N) + N^{-({u}-p)/{u}} & \text{ if } p=n/2 \text{ and } {u}\ne 2p,\\
N^{-p/n} + N^{-({u}-p)/{u}} & \text{ if } p<n/2 \text{ and } {u}\ne \frac{n}{n-p},
\end{cases}
\end{multline}
where $n=\text{dim}(\Xc)$, and $C$ is a constant depending only on $p$, ${u}$, and $n$. If the number $N$ of particles
sampled from each ${Q}_t(\cdot|z_t^s)$ is the same for all  $s=1,\dots,M_t$, the expected distance
$\Eb \big[ \Wc_p^{\tilde{\lambda}_t}(\widehat{Q}_t,Q_t) \big]$ is bounded by the expression \eqref{fournier} as well.

Our procedure adds to this error a fully controllable part $\Wc_p^{\tilde{\lambda}_t}(\widetilde{Q}_t,\widehat{Q}_t)$
by constructing a set of representative points $z_{t+1}^j$, $j=1,\dots,M_{t+1}$, each of which may serve as a ``descendant''
of multiple points $z_t^s$. Our experience shows that for large $M_t$ the total number of these points, $M_{t+1}$, is comparable to
$M_t$, and thus much smaller than the number of particles $NM_t$. As a result, the total number of representative points, while still exponential in the dimension of the state space, grows only linearly with the number of time steps.
We elaborate on it in the next section.

 }

\section{Numerical Illustration}
\label{s:illustration}

Consider $n$ stocks $\big\{S_t^{(i)}\big\}$, $i=1, \ldots, n$, in an arbitrage-free and complete market, following (under the risk-neutral probability measure ${\Qb}$) the equations:
\begin{equation}
\D S_t^{(i)}=r S_t^{(i)} \D t+\sigma^{(i)} S_t^{(i)} \D W_t^{{\Qb}}, \quad i=1,\dots,n,\quad t \in[0, T].
\end{equation}
Here, $\{W_t^{{\Qb}}\}$ is an $n$-dimensional Brownian motion under $\Qb$, $r$ is the risk-free interest rate, and $\sigma^{(i)}$ is the $n$ dimensional (row) vector of volatility coefficients of stock $i$. We assume that the coefficients $r$ and $\sigma$ are constant, but our methodology is applicable to problems with varying coefficients as well.

An option is one of the most common financial derivatives that give buyers the right, but not the obligation, to buy or sell an underlying asset at an agreed-upon price during a certain period of time. The American option is the type of option that can be exercised anytime, prior to the maturity time $T$.
If exercised at time $t$, the option pays $\varPhi(S_{t})$ for some known function $\varPhi:{\Rb}^n\to [0,+\infty)$. The price of the American option is given by the optimal value of the stopping problem:
\begin{equation}
V_t(x)=\sup_{{\tau - \text{stopping time}}\atop {t \le \tau \le T}} {E}^{\Qb}\big[e^{-r (\tau-t)} \varPhi\left(S_\tau\right) \big| S_{t}=x \big], \quad x \in {R}^n,
\end{equation}
In our example, $\varPhi(S_{t}) = \max\Big(0,K - \sum_{i=1}^n w_iS_t^{(i)}\Big)$ is the value of the basket put option, with the basket weights $w_i$, $i=1,\dots,n$.

To develop a numerical scheme for approximating the option value, we first partition the time interval $[0, T]$ into short intervals of length $\varDelta t= T/N$:
$\varGamma_N=\big\{t_{i}=i\varDelta t:$ $i=0,1,\dots,N\big\}$.

With the exercise times restricted to $\varGamma_N$, we approximate the option value by
\begin{equation}
V_t^{(N)}(x)=\sup_{{\tau - \text{stopping time}}\atop{\tau \in \Gamma_N}} {E}^{\Qb}\big[e^{-r (\tau-t)} \varPhi\left(S_\tau\right) \big|  S_{t}=x \big], \quad t\in \varGamma_N,\quad x \in {R}^n.
\end{equation}
We view  $V_t^{(N)}(x)$ as an approximation to the actual American option price when $N$ increases to infinity.
It satisfies the following dynamic programming equations:
\begin{gather*}
V_{t_N}^{(N)}(x) =  \varPhi(x), \quad x\in \Rb^n,\\
V_{t_i}^{(N)}(x) =\max \left\{\varPhi(x), {E}^{{\Qb}}\big[e^{-r\varDelta t} V_{t_{i+1}}^{(N)}\left(S_{t_{i+1}}\right) \big| S_{t_i}=x\big]\right\},\quad i=0,1,\dots,N-1,
\end{gather*}
which is a special case of \eqref{sigma-stopping}.
We apply two methods to simulate the movements of stocks and compare the values of the approximation of the American basket option. The first method is the grid point selection method based on the integrated transportation distance. For every time step $t_i$, we select the representative point(s) $z_i^j, j = 1, \ldots, M_i$ to represent the state space at time ${t_i}$, as outlined in \S \ref{s:particles}. We compare the above method with the binomial tree method, a lattice method based on the random walk approximation to the Brownian motion. Between the start and expiration dates, each grid point in a lattice represents the state of the system at a given time step. Starting from the grid points at the final time step, the prices at the preceding grid points are computed in a backward direction. Since every node of the lattice has $2^n$ descendants, the number of lattice points in the binomial tree method grows exponentially,
as the number of the time steps increases. In the grid point selection method, the total number of representative points
grows approximately at a linear rate with respect to the total number of time steps $N$.

\begin{table}[h]
\caption{Convergence of the American basket put option prices with respect to the number of time discretization steps.}
\begin{center}
\begin{tabular}{ ccccc}
\toprule
 $N$ &  grid  &binomial  \\
 \midrule
 1 & 0.832  & 0.824 \\
 2 & 0.869 & 1.009  \\
 5 & 0.880  & 0.896 \\
 10 & 0.880 & 0.873  \\
 25 & 0.884  & 0.887 \\
 50 & 0.887  &  0.889 \\
    \bottomrule
\end{tabular}
\end{center}
 \label{2d-table}
\end{table}

In the initial experiment, both methods are applied to evaluate the American basket put option with $n = 2$ and the payoff function for the American basket put is $\varPhi_p(S_{t}) = \max (K-\sum_{i=1}^n w_i S_t^{(i)}, 0)$, where $w_i$ is the percentage of stock $i$ held in the portfolio and $K$ is the strike price. The values of the parameters are
$S_0=[10,10]$, $r=0.03$, $K=10$, $w=(0.5,0.5)$, and $T=1$. The volatility coefficients were:
$\sigma=\left[\begin{array}{cccc}
0.5 & -0.2; &
-0.2 & 0.5
\end{array}\right]$.

Table \ref{2d-table} compares the approximated option prices using the grid point selection method and the binomial tree method. Figure \ref{fig:abp} summarizes the convergence of the American basket put option as the number of time steps increases.
Moreover, the upper bound of the error in estimating value function is determined by the integrated transportation distance at every time stage. For the grid point selection method, we have computed the the integrated transportation distances for the first few time stages. $\varDelta_0 = 0.239$, $\varDelta_1 = 0.211$, $\varDelta_2 = 0.192$, $\varDelta_3 = 0.190$, and $\varDelta_4 = 0.181$.


    \begin{figure}
    \centering

    \begin{minipage}[b]{0.45\textwidth}
        \centering
        \includegraphics[width=\textwidth]{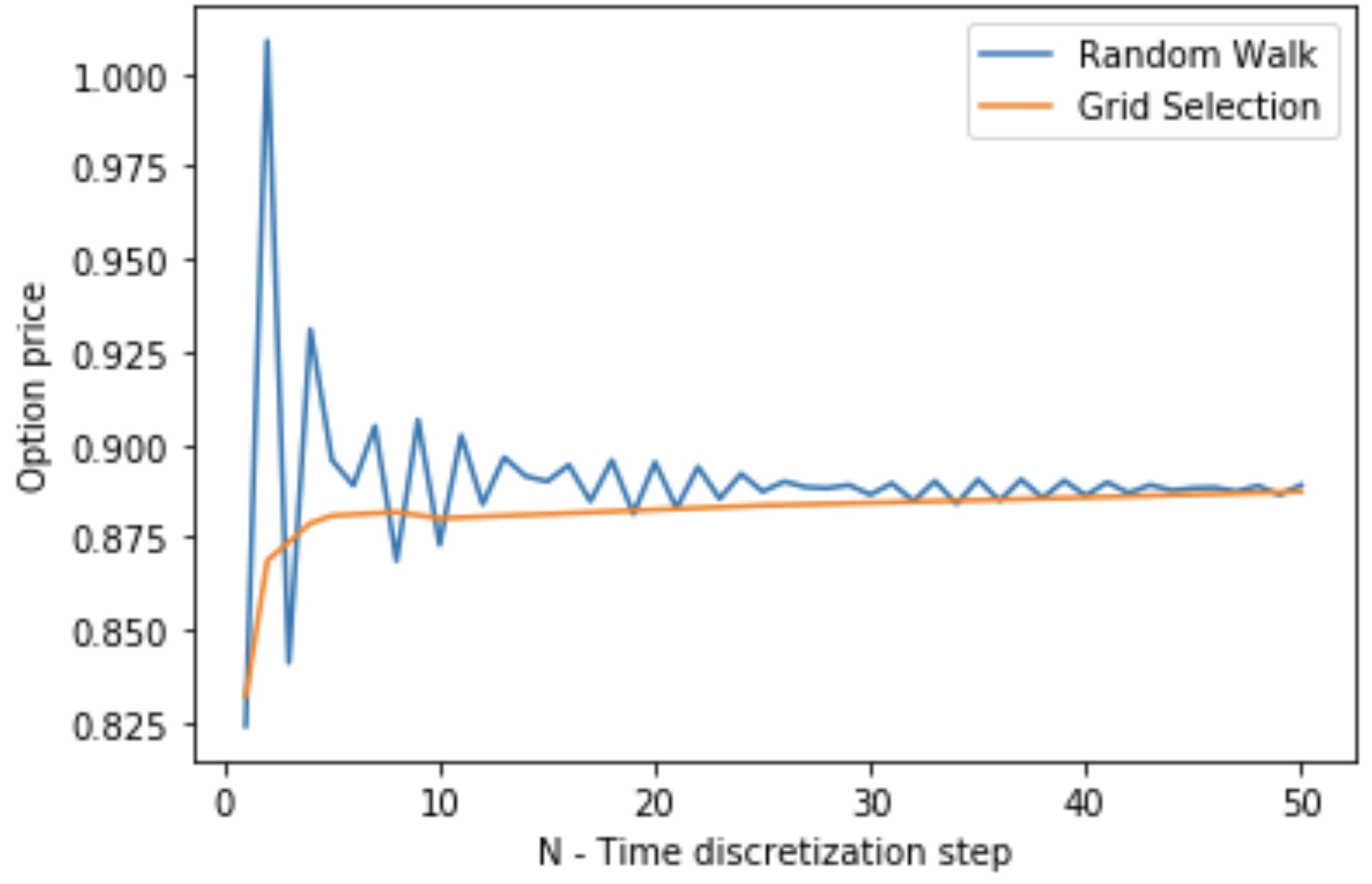}
 \caption{The approximate value of the American basket put option as a function of the number of time steps}
 \label{fig:abp}
    \end{minipage}
    \hfill
    \begin{minipage}[b]{0.45\textwidth}
        \centering
        \includegraphics[width=\textwidth]{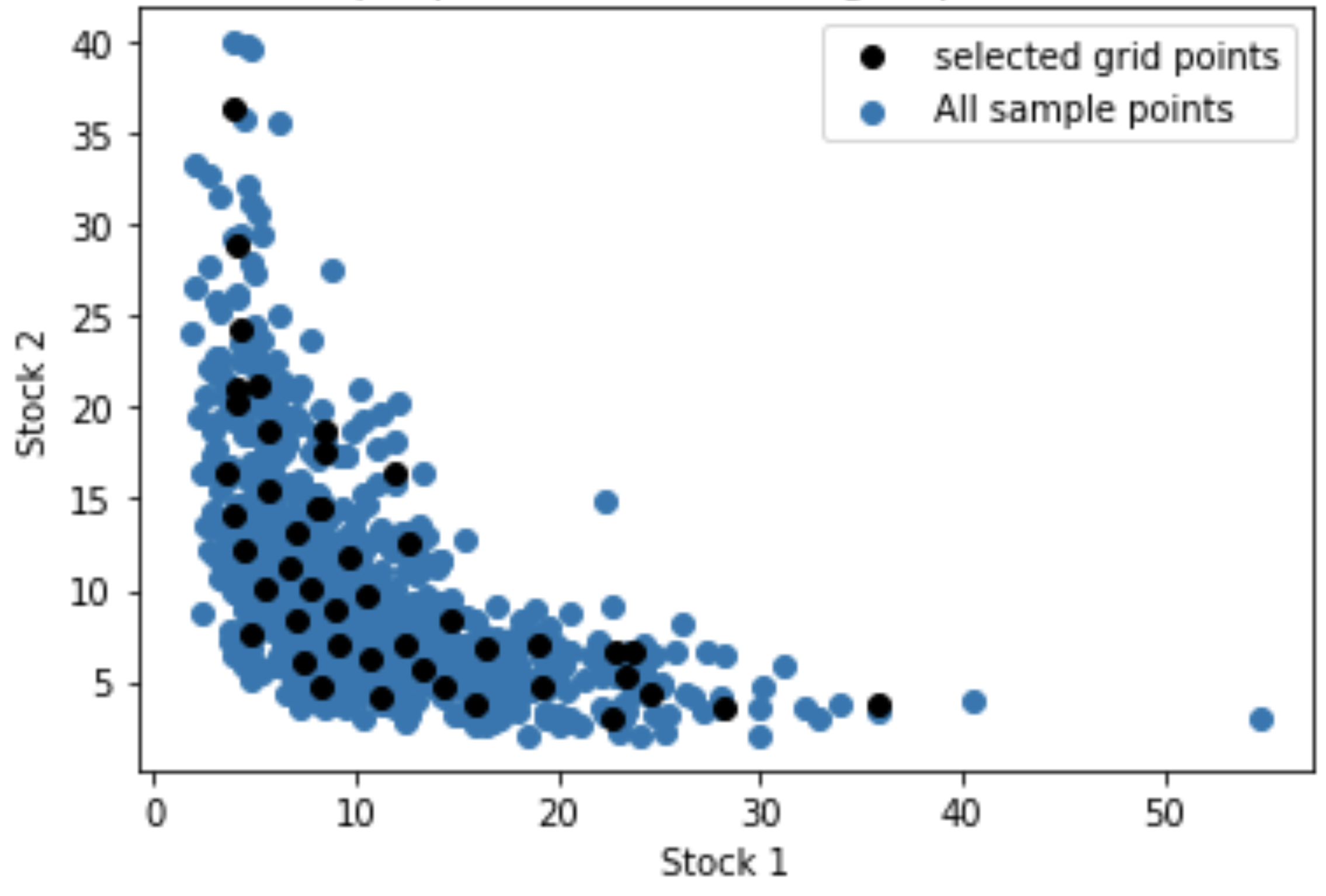}
 \caption{All sample points (blue, N = 1000) vs. selected grid points (black, M= 46).}
 \label{fig:S1}
    \end{minipage}

\end{figure}


In order to demonstrate the stability of the approximate prices using our grid point selection method, we will also apply this method on risk measures at $T = 1.$ A practically relevant law-invariant coherent measure of risk is the mean--semideviation of order $p \geq 1$, defined in \eqref{msd-form}.
Figure \ref{fig:S1} illustrates an example of selecting grid points from the simulated stock prices at $T=1$.
In the grid selection method, we set the number of grid points to be around 400 selected out of 1000 randomly sampled points. We repeated the experiment over 900 times and recorded the mean and semideviation estimates. In the Monte Carlo experiment, we sampled 1000 points and evaluated plug-in estimates of the mean and the semideviation; this experiment was repeated 5000 times.
In Figure \ref{fig:esp}, we plot the histograms of the estimated expected values, and in Figure \ref{fig:ems}, the histograms of the estimated semideviations. It is obvious that the approximated values from the grid selection method are more stable than those from the Monte Carlo simulation.

\begin{figure}
\centering

\begin{subfigure}{0.45\textwidth}
    \includegraphics[width=\textwidth]{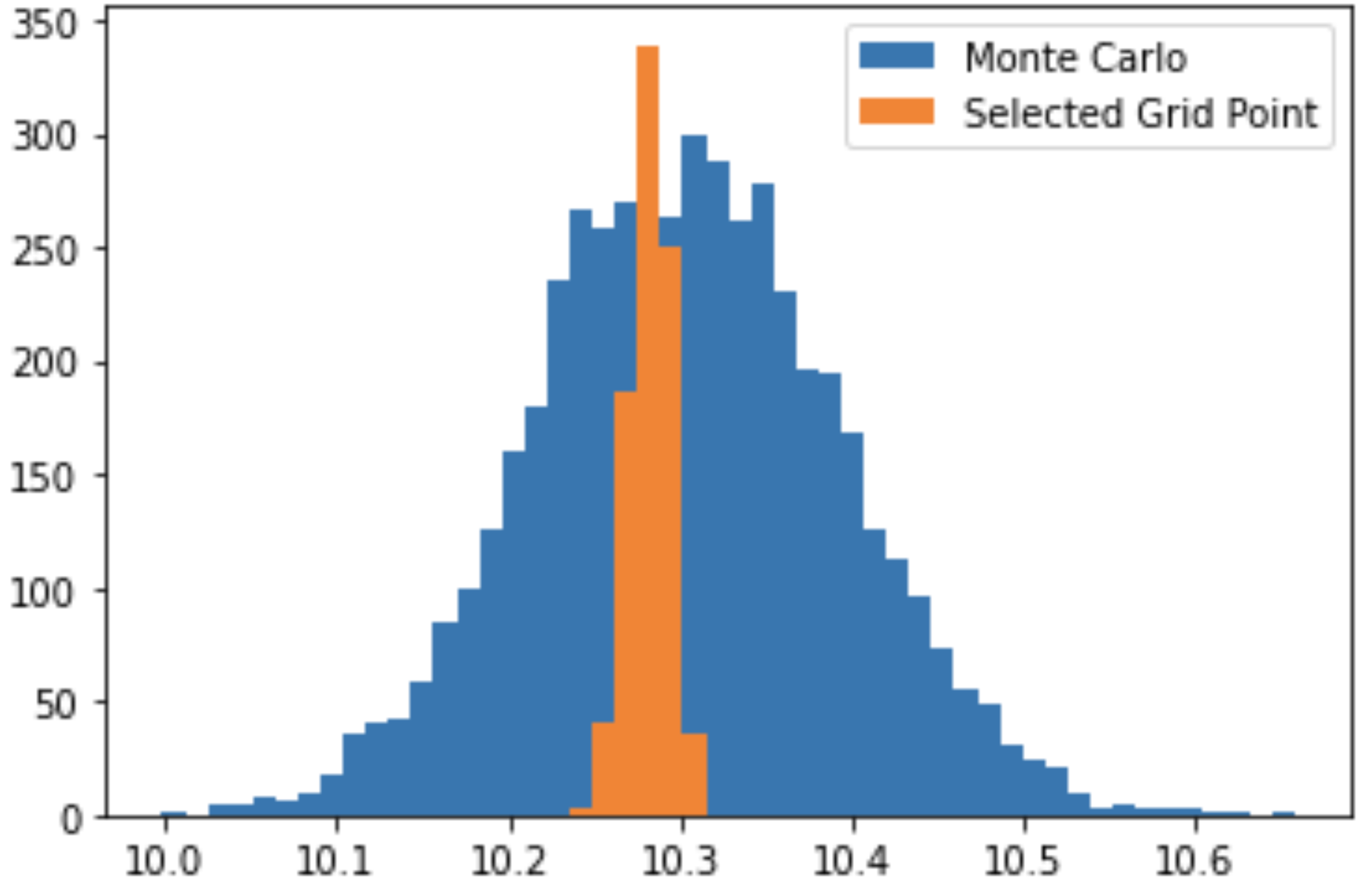}
 \caption{Histogram of the estimated expected value of the stock price.}
 \label{fig:esp}
\end{subfigure}
\hfill
\begin{subfigure}{0.45\textwidth}
    \includegraphics[width=\textwidth]{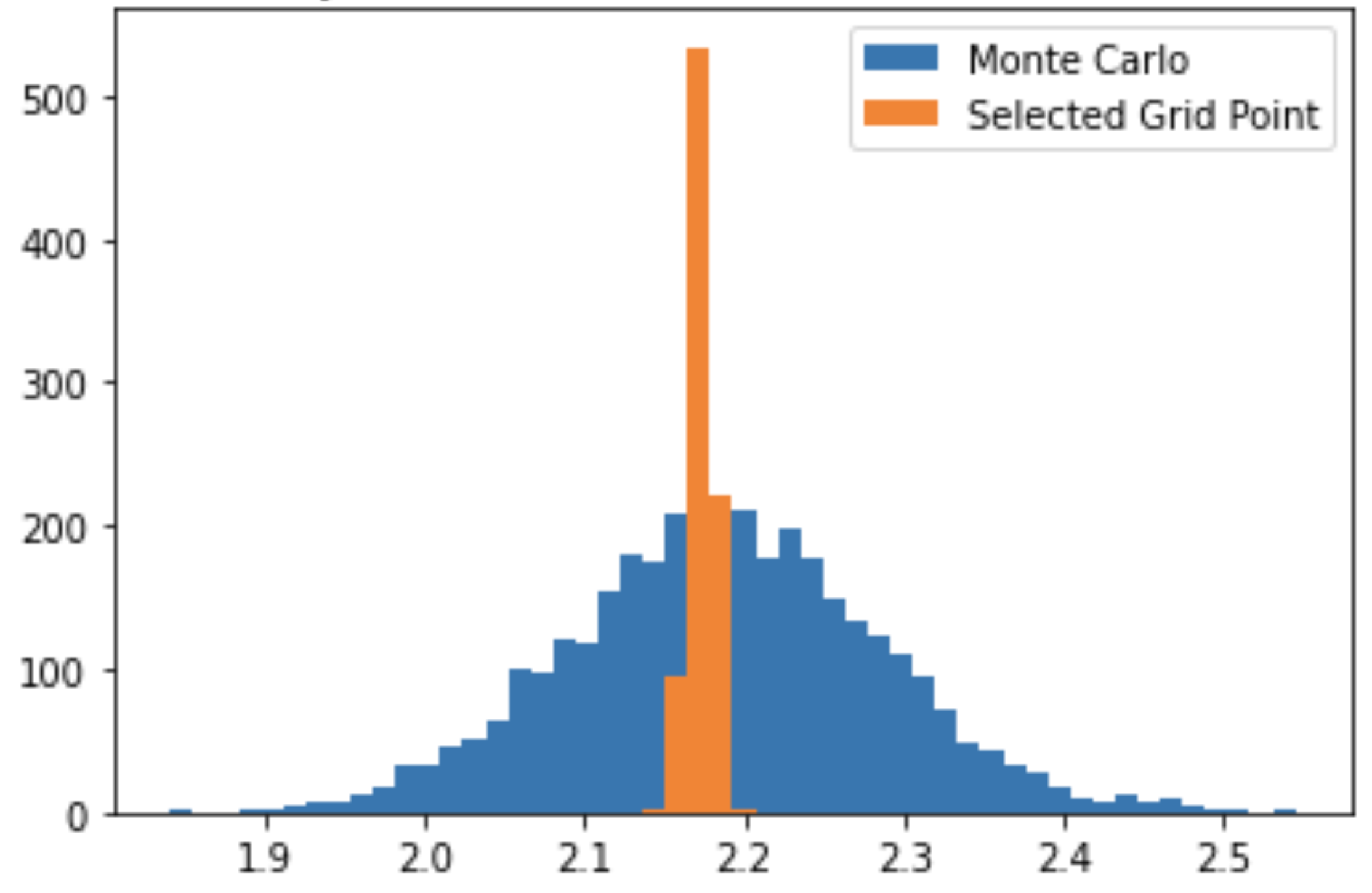}
 \caption{Histogram of the estimated semideviation of the stock price.}
 \label{fig:ems}

\end{subfigure}
\hfill

\caption{Monte Carlo simulation vs. the grid selection method. }

\end{figure}

In our more challenging experiment, we estimated the American put option value for a five-dimensional stock basket.
The values of the parameters are
$S_0=[10,10,10,10,10]$, $r=0.03$, $K=10$, $w=(0.2,0.2,0.2,0.2,0,2)$, $T=1$, and
\[
\sigma=\left[\begin{array}{ccccc}
0.5 & 0.2 & 0.3 & -0.2 & 0.15 \\
0.2 & 0.5 & -0.15 & 0.3 & 0.12 \\
0.3 & -0.15 & 0.75 & -0.1 & 0.1 \\
-0.2 & 0.03 & -0.1 & 0.3 & 0.05 \\
0.15 & 0.12 & 0.1 & 0.05 & 0.4
\end{array}\right].
\]

Table \ref{GP} displays the convergence of the American put option prices as we increase the number $N$ of time discretization points, using the grid selection method and the binomial tree method. $M$ refers to the total number of grid points used. As shown in the table, the binomial tree method cannot go beyond $N = 12$ because the total number of grid points, $M$, increases exponentially with $N$. The grid point selection method achieves similar results to that of the binomial tree method while requiring {   only {linear} growth of the total number of representative points
with the number of stages.} \vspace{-1ex}

\begin{table}[H]
 \caption{ Convergence of the estimates of the American put price with respect to the number of time discretization steps $N$.}
\begin{center}
\begin{tabular}{ ccccc}
 \toprule
 $N$ &  put - grid & put - binomial & $M$ - grid & $M$ - binomial \\
 \midrule
 1 & 1.168 & 1.179 & 30343 & 33\\
 2 & 1.188 & 1.223 & 38740 & 276\\
 3 & 1.207 & 1.239 & 50891 & 1300\\
 4 & 1.213 & 1.240 & 56970 & 4425\\
 5 & 1.231 & 1.241 & 74044 & 12201\\
 6 & 1.231 & 1.242 & 81022 & 29008\\
 7 & 1.240 & 1.242 & 94592 & 61776\\
 8 & 1.239 & 1.244 & 97639 & 120825\\
 9 & 1.250 & 1.244 & 127981 & 220825\\
 10 & 1.254 & 1.244 & 136378 & 381876\\
 11 & 1.258 & 1.245 & 148528 & 630708\\
 12 & 1.259 & 1.246 & 154607 & 1002001\\
  \bottomrule
\end{tabular}
\end{center}
\label{GP}
\end{table}


\bibliographystyle{abbrv}
\newcommand{\noop}[1]{}

\appendix

\section{Proofs of the statements in \S \ref{s:kernel-distance} }
\begin{proof}[Proof of Theorem \ref{TS_T}]

It is obvious that $\Wc_p^\lambda(Q, \widetilde{Q}) \geq 0$ for any $Q, \widetilde{Q} \in \Qc_{p}^\lambda(\Xc,\Yc)$ and $\Wc_p^\lambda(Q, \widetilde{Q})=0$ if and only if $Q= \widetilde{Q}$ ${\lambda}$-a.s.. We next verify the triangle inequality.
For all $Q, {Q}^{\prime},  \widetilde{Q} \in \Qc_p^\lambda(\Xc,\Yc)$, by the triangle inequality for $W_p(\cdot,\cdot)$ and then by the Minkowski inequality,
we obtain
\begin{equation*}
\begin{split}
\lefteqn{\Wc_p^\lambda(Q, \widetilde{Q})
\leq \left(\int_{\Xc}\big[{W}_{p}(Q(\cdot|x), Q^{\prime}(\cdot|x))+{W}_{p}(Q^{\prime}(\cdot|x), \widetilde{Q}(\cdot|x))\big]^p\; {\lambda}(\D x) \right)^{1/p}} \\
&\leq \left(\int_{\Xc}\big[{W}_{p}(Q(\cdot|x), Q^{\prime}(\cdot|x))\big]^p\; {\lambda}(\D x) \right)^{1/p} +
\left(\int_{\Xc}\big[{W}_{p}(Q^{\prime}(\cdot|x), \widetilde{Q}(\cdot|x))\big]^p\; {\lambda}(\D x) \right)^{1/p}\\
&=\Wc_p^\lambda(Q, {Q}^{\prime}) + \Wc_p^\lambda({Q}^{\prime},\widetilde{Q}).
\end{split}
\end{equation*}
Furthermore, setting $Q'(\cdot|x)=\delta_{\{y_0\}}(\cdot)$ and using \eqref{kernel-class}, we get
\begin{equation}
\label{dist-to-delta}
\begin{aligned}
\lefteqn{\big[\Wc_p^\lambda(Q,\delta_{\{y_0\}})\big]^p
= \int_{\Xc} \big[{W}_{p}(Q(\cdot | x),\delta_{\{y_0\}} )\big]^p \;{\lambda}(\D x)}\quad \\
&=\int_{\Xc} \int_{\Yc} d(y, y_0)^{p} \;  Q({\D y}|x) \;{\lambda}(\D x)
\le C(Q)\int_{\Xc} \big(1 + d(x,x_0)^p\big) \;{\lambda}(\D x) < \infty,
\end{aligned}
\end{equation}
which proves
the finiteness of $\Wc_p^\lambda(Q, \widetilde{Q})$, if $\lambda \in \Pc_p(\Xc)$.
\end{proof}

\begin{proof}[Proof of Theorem \ref{TS_T2}]
From Eq \eqref{TS} we obtain,
\begin{equation*}
\begin{split}
    \big[\Wc_{p}^\lambda(Q, \widetilde{Q}) \big]^p
    &= \int_{\Xc}  \int_{\Yc \times \Yc } d(y, y')^{p} \; \pi^*({\D y}, {\D y}'|x)\; \lambda(\D x),
\end{split}
\end{equation*}
where $\pi^*(\cdot,\cdot|x)$ is the optimal transportation plan between $Q(\cdot|x)$ and $\widetilde{Q}(\cdot|x)$. By the measurable selection theorem,
the mapping $x \mapsto \pi^*(\cdot,\cdot|x)$ may be viewed as a kernel from $\Xc$ to $\Pc(\Yc\times\Yc)$.

Now, we construct from $\pi^*$ a transportation plan $\varPi^*\in \Pc\big((\Xc\times \Yc) \times (\Xc\times \Yc)\big)$: for all $A_X,B_X\in \Bc(\Xc)$ and
all $A_Y,B_Y\in \Bc(\Yc)$ we set
\begin{equation}
\label{pi-star}
\varPi^*\big( (A_X\times A_Y)\times (B_X\times B_Y)\big) = \int_{A_X\cap B_X} \pi^*(A_Y \times B_Y|x)\;\lambda({\D x}).
\end{equation}
Setting $B_X = \Xc$ and $B_Y=\Yc$ we obtain the marginal of $\varPi^*$:
\begin{multline*}
\varPi^*\big( (A_X\times A_Y)\times (\Xc\times \Yc)\big) = \int_{A_X} \pi^*(A_Y \times \Yc|x)\;\lambda({\D x}) \\
= \int_{A_X} Q(A_Y|x)\;\lambda({\D x}) = (\lambda \circledast Q) (A_X \times A_Y).
\end{multline*}
The second marginal is verified in an analogous way and thus $\varPi^*$ moves $\lambda \circledast Q$ to $\lambda \circledast \widetilde{Q}$.
Then, by virtue of \eqref{pi-star},
\begin{align*}
{W}_p({\lambda} \circledast Q,{\lambda} \circledast \widetilde{Q})^p
&\le \int_{(\Xc \times \Yc) \times (\Xc \times \Yc)} d\big( (x,y), (x',y')\big)^p\; \varPi^*({\D x} \;{\D y},{\D x}'\,{\D y}')\\
& = \int_{\Yc \times \Yc} d(y,y')^p \int_{\Xc} \pi^*({\D y},{\D y}'|x)\;\lambda({\D x}) = \big[\Wc_{p}^\lambda(Q, \widetilde{Q}) \big]^p,
\end{align*}
which verifies the  left inequality in \eqref{hierarchy}.

Next, for the optimal transportation plan $\widehat{\varPi}\in \Pc\big((\Xc\times \Yc) \times (\Xc\times \Yc)\big)$, with marginals
$\lambda \circledast Q$ and $\lambda \circledast \widetilde{Q}$,
we construct a transportation
plan $\hat{\pi}\in \Pc(\Yc\times\Yc)$ as
\[
\hat{\pi}(A_Y \times B_Y) = \widehat{\varPi}\big( (\Xc\times A_Y)\times (\Xc \times B_Y)\big), \quad \forall\, A_Y,B_Y\in \Bc(\Yc).
\]
Then
\[
\hat{\pi}(A_Y \times \Yc) = \widehat{\varPi}\big( (\Xc\times A_Y)\times (\Xc \times \Yc)\big) = [\lambda \circledast Q ](\Xc\times A_Y) = [\lambda \circ Q ](A_Y).
\]
The second marginal is verified analogously and thus $\hat{\pi}$ moves $\lambda \circ Q$ to $\lambda \circ \widetilde{Q}$. Therefore,
\begin{align*}
\lefteqn{{W}_p({\lambda} \circ Q,{\lambda} \circ \widetilde{Q})^p
\le \int_{ \Yc \times \Yc} d(y, y')^p\; \hat{\pi}({\D y},{\D y}')}\quad \\
&= \int_{(\Xc \times \Yc) \times (\Xc \times \Yc)} d(y,,y')^p\; \widehat{\varPi}({\D x} \,{\D y},{\D x}'\,{\D y}')\\
&\le \int_{(\Xc \times \Yc) \times (\Xc \times \Yc)} d\big( (x,y), (x',y')\big)^p\; \widehat{\varPi}({\D x} \,{\D y},{\D x}'\,{\D y}')
= {W}_p({\lambda} \circledast Q,{\lambda} \circledast \widetilde{Q})^p.
\end{align*}
which is the  right inequality in \eqref{hierarchy}.
\end{proof}
\begin{proof}[Proof of Theorem \ref{t:kernel-distance-metric}]
The implication (ii)$\Rightarrow$(i) follows from Theorem \ref{TS_T2}, because the first inequality in \eqref{hierarchy} yields ${W}_p({\lambda} \circledast Q_k,{\lambda} \circledast {Q})\to 0$,
and thus $\lambda \circledast Q_k \overset{p}\to {\lambda} \circledast {Q}$, by virtue of \cite[Thm. 6.9]{villani2009optimal}.
The latter convergence implies that Definition \ref{d:kernel-convergence} is satisfied.

 To prove the implication (i)$\Rightarrow$(ii), we adopt some ideas of the proof of \cite[Thm 6.9]{villani2009optimal}.
From Eq. \eqref{TS} we obtain,
\begin{equation*}
\begin{split}
    \big[\Wc_{p}^\lambda(Q_k,{Q}) \big]^p
    &= \int_{\Xc}  \int_{\Yc \times \Yc } d(y, y')^{p} \; \pi_k({\D y}, {\D y}'|x)\; \lambda(\D x),
\end{split}
\end{equation*}
where $\pi_k(\cdot,\cdot|x)$ is the optimal transport plan between $Q_k(\cdot|x)$ and ${Q}(\cdot|x)$. By the measurable selection theorem,
the mapping $x \mapsto \pi_k(\cdot,\cdot|x)$ may be viewed as a kernel from $\Xc$ to $\Pc(\Yc\times\Yc)$.
Since Definition \ref{d:kernel-convergence} implies that $\lambda \circledast Q_k\rightharpoonup \lambda \circledast Q$,
it follows that $Q_k(\cdot|x) \rightharpoonup Q(\cdot|x)$  for $\lambda$-almost all $x\in \Xc$. For every such $x$, by virtue of the Prohorov theorem,
the sequence $\{Q_k(\cdot|x)\}$ is tight,
and thus the sequence $\{\pi_k(\cdot,\cdot|x)\}$ is tight as well \cite[Lem. 4.4]{villani2009optimal}. By passing to a subsequence, if necessary, we conclude that the sequence $\{\pi_k(\cdot,\cdot|x)\}$
is weakly convergent to some limit $\{\pi^*(\cdot,\cdot|x)\}$. The limit must be the optimal transport from $Q(\cdot|x)$ to itself:
$Q(\cdot|x)\circ \Ib$, where $\Ib$ is the identity kernel $y\mapsto \delta_{y}$. It follows that the limit does not depend on the subsequence;
the entire sequence $\{\pi_k(\cdot,\cdot|x)\}$ is weakly convergent to $\pi^*(\cdot,\cdot|x)$, for $\lambda$-almost all $x$.

For any $R>0$, we have a simple upper bound:
\begin{align*}
 \big[\Wc_{p}^\lambda(Q_k,{Q}) \big]^p
 &\le \int_{\Xc}  \int_{\Yc \times \Yc } \big[d(y, y')\wedge R\big]^{p} \; \pi_k({\D y}, {\D y}'|x)\; \lambda(\D x) \\
 &{\ } + \int_{\Xc}  \int_{\Yc \times \Yc } \big[d(y,y')^p -R^p\big]_+ \; \pi_k({\D y}, {\D y}'|x)\; \lambda(\D x).
 \end{align*}
Using the inequality
\[
\big[d(y,y')^p -R^p\big]_+ \le 2^p d(y,y_0)^p \1_{\{d(y,y_0) \ge R/2\}}+2^p d(y_0,y')^p \1_{\{d(y_0,y') \ge R/2\}},
\]
we can continue the upper bound as follows:
 \begin{align*}
\lefteqn{  \big[\Wc_{p}^\lambda(Q_k,{Q}) \big]^p
\le \int_{\Xc}  \int_{\Yc \times \Yc } \big[d(y, y')\wedge R\big]^{p} \; \pi_k({\D y}, {\D y}'|x)\; \lambda(\D x) }\\
&{\ } + 2^p \hspace{-1em} \int\limits_{\{d(y,y_0) \ge R/2\}} \hspace{-1em} d(y,y_0)^p \; \pi_k({\D y}, {\D y}'|x)\; \lambda(\D x)
 + 2^p \hspace{-1em}\int\limits_{\{d(y_0,y') \ge R/2\}} \hspace{-1em} d(y_0,y')^p \; \pi_k({\D y}, {\D y}'|x)\; \lambda(\D x) \\
&= \int_{\Xc}  \int_{\Yc \times \Yc } \big[d(y, y')\wedge R\big]^{p} \; \pi_k({\D y}, {\D y}'|x)\; \lambda(\D x) \\
&{\ } + 2^p \hspace{-1em} \int\limits_{\{d(y,y_0) \ge R/2\}} \hspace{-1em} d(y,y_0)^p \; Q_k({\D y}|x)\; \lambda(\D x)
 + 2^p \hspace{-1em}\int\limits_{\{d(y_0,y') \ge R/2\}} \hspace{-1em} d(y_0,y')^p \; Q({\D y}'|x)\; \lambda(\D x).
\end{align*}
As the sequence $\{\pi_k(\cdot,\cdot|x)\}$ converges weakly  to $\pi^*(\cdot,\cdot|x)$ for $\lambda$-almost all $x$, the first term on the
right-hand side converges to 0, for every $R>0$. Furthermore, by Definition \ref{d:villani-def-1}(ii),
since $\lambda\circ Q_k \overset{p}\to \lambda\circ Q$,
\[
\lim_{R\to \infty} \limsup_{k\to\infty} \int_{\{d(y,y_0) \ge R/2\}}d(y,y_0)^p \; Q_k({\D y}|x)\; \lambda(\D x) = 0.
\]
The same is true for the third term. Putting these estimates together, we conclude that
$\lim_{k\to\infty} \Wc_{p}^\lambda(Q_k,{Q})   = 0$.
\end{proof}
{
\begin{proof}[Proof of Theorem \ref{t:KR-kernels}]
Theorem \ref{t:KR} implies that for all $f\in F$
\begin{align*}
\Wc_1^\lambda(Q,\widetilde{Q}) &= \int_{\Xc} {W}_{1}(Q(\cdot | x), \widetilde{Q}(\cdot | x)) \;{\lambda}(\D x)\\
&\ge \int_{\Xc}
\left\{\int_{\mathcal{Y}} f(x,y)\; Q({\D y}|x)-\int_{\mathcal{Y}}f(x,y) \; \widetilde{Q}({\D y}|x) \right\}
\;{\lambda}(\D x)\\
&=\int_{\mathcal{X}\times \mathcal{Y}} f(x,y)\; (\lambda \circledast Q)(\D x\;\D y)-\int_{\mathcal{X}\times \mathcal{Y}} f(x,y) \;
(\lambda \circledast \widetilde{Q})(\D x\;\D y).
\end{align*}
This verifies the inequality ``$\ge$'' in \eqref{KR-kernels}. To verify the reverse inequality, let $\varepsilon>0$ and define the multifunction
$F_\varepsilon:\Xc \rightrightarrows \text{Lip}(\Yc,\Rb)$ as follows
\begin{multline*}
F_\varepsilon(x) = \Big\{ \psi\in  \text{Lip}(\Yc,\Rb): \|\psi\|_{\text {\rm Lip}} \leq 1,\\
\int_{\mathcal{Y}} \psi(y)\; Q({\D y}|x)-\int_{\mathcal{Y}}\psi(y)\; \widetilde{Q}({\D y}|x)
\ge{W}_{1}(Q(\cdot | x), \widetilde{Q}(\cdot | x)) - \varepsilon \Big\}, \quad x\in \Xc.
\end{multline*}
It is measurable and, owing to Theorem \ref{t:KR}, has nonempty closed values.
Therefore, by the measurable selection theorem, a selector $\varPsi_\epsilon: \Xc \to \text{Lip}(\Yc,\Rb)$ exists,
such that $\varPsi_\varepsilon(x)\in F_\varepsilon(x)$ for all $x\in \Xc$. Define $f_\varepsilon(x,y) = \big[\varPsi_\varepsilon(x)\big](y) $, $x \in \Xc$, $y\in \Yc$. By construction,  $f_\epsilon \in F$ and
\begin{align*}
 \Wc_1^\lambda(Q,\widetilde{Q}) &\le \int_{\Xc}
 \bigg\{\int_{\mathcal{Y}} f_\varepsilon(x,y)\; Q({\D y}|x)-\int_{\mathcal{Y}}f_\varepsilon(x,y)\; \widetilde{Q}({\D y}|x) + \varepsilon  \bigg\}
 \;{\lambda}(\D x)\\
  &\le \sup_{f(\cdot,\cdot)\in F}\left\{\int_{\mathcal{X}\times \mathcal{Y}} f(x,y)\; (\lambda \circledast Q)(\D x\,\D y)-\int_{\mathcal{X}\times \mathcal{Y}} f(x,y) (\lambda \circledast \widetilde{Q})(\D x\,\D y)\right\} + \varepsilon.
\end{align*}
Since $\varepsilon>0$ was arbitrary, the inequality ``$\le$'' (and then the equality) in  \eqref{KR-kernels} is true. As subtracting $f(\cdot,y_0)$ from $f(\cdot,\cdot)$ does not affect the right-hand side of \eqref{KR-kernels}, we may restrict $F$ to contain only the functions whose value at $y_0$ is 0.
\end{proof}
}
\section{Comparison of Kernel Distances on Gaussian Mixture Models}

In this section, we consider Gaussian mixture models with varying dimensions and numbers of centers, each having a different weight (marginal probability). We denote by ${\Xc}_0$ the set of the centers, and by $\lambda_0$
the marginal distribution.

In each example, we select grid points from the same set of sample points. The point selection process employs two metrics:
$\Db_{1}(Q, \widetilde{Q})=\sup_{x\in \Xc_0} {W}_{1}(Q(\cdot | x), \widetilde{Q}(\cdot | x))$,
and the integrated transportation distance,
$\Wc_1^{\lambda_0}(Q,\widetilde{Q}) = \sum_{x\in\Xc_0} \lambda_0(x)\, W_1(Q(\cdot|x),\widetilde{Q}(\cdot|x))$.
The number of points to be selected by both methods is the same.

Table~\ref{gmm} presents the dimensions of the mixture model (dim), the number of centers (center), the number of particles sampled from each center (particles), the number of selected particles (selected), the solution times for both methods (in seconds), and the corresponding Wasserstein distance $W_1$ of the selected points to the particle distribution.  For the sake of simplicity, we  refer to the supremum distance as ``sup'' and the integrated transportation distance as ``ITD'' in the table header. The selection algorithm utilizing the integrated transportation distance consistently achieves a lower $W_1$ distance and faster execution time in all examples.\vspace{-1ex}

\begin{table}[H]

\caption{Comparison of the supremum distance and the integrated transportation distance}

\begin{center}
\label{gmm}
\begin{tabular}{ cccccccc}
\toprule
 dim & centers & particles & selected  & sup (s) & ITD (s) & sup $W_1$ & ITD  $W_1$\\
    \midrule
2 & 5 & 400 & 100 & 1329.27 & 1320.15 & 0.288 &0.268 \\
2 & 10 &200 & 100 & 1426.99 & 1296.43 & 0.466 & 0.457 \\
2 & 16 & 160 & 128 & 1365.93 & 812.32 & 0.645 & 0.604\\
3 & 3 & 500 & 375 & 1296.96 &530.36 & 0.913 & 0.901 \\
3 & 5 & 400 & 500 &1931.75 & 1253.03 & 0.953 & 0.784 \\
5 & 3 & 800 & 600 & 1683.79 & 1235.43  & 1.963 & 1.812 \\
    \bottomrule
\end{tabular}
\end{center}

\end{table}

In Figures \ref{fig:P1}--\ref{fig:P3}, the subfigures (a) and (b) illustrate the sample points and the grid points ${z^k}$ (represented by black dots) selected using the supremum distance and the integrated transportation distance, respectively, for the three two-dimensional examples. The sample points ${x^{si}}$ are depicted in different colors to represent the various Gaussian distributions.

In all experiments, the integrated transportation distance model was solved faster and resulted in a more accurate representation of the mixture distribution.
In experiments with problems of higher dimension
these differences were dramatic.

All numerical results were obtained using Python (Version 3.7) on a Macintosh HD laptop with a 2.9 GHz CPU and 16GB memory. The data are available in the working paper version.

\begin{figure}[H]
\centering

\begin{subfigure}{0.45\textwidth}
    \includegraphics[width=\textwidth]{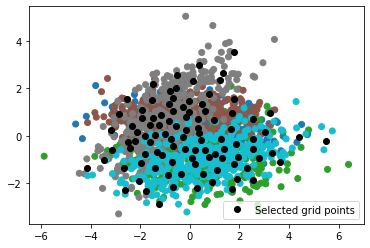}
  \caption{The supremum distance selection.}
  \label{fig:sup2}
\end{subfigure}
\hfill
\begin{subfigure}{0.45\textwidth}
    \includegraphics[width=\textwidth]{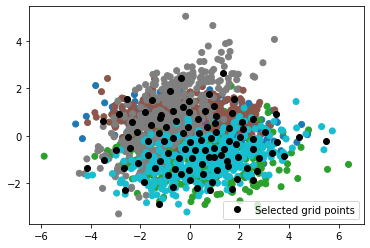}
  \caption{The ITD selection.}
  \label{fig:kernel2}
\end{subfigure}
\vspace{-2ex}
\caption{Gaussian Mixture model with 5 centers and samples of 400 drawn from each center; $\text{dim}(\beta) = 1000000$, $\text{dim}(\gamma) = 500$, and 100 selected representative points.}
\label{fig:P1}
\end{figure}

\vspace{-5ex}

\begin{figure}[H]
\centering

\begin{subfigure}{0.45\textwidth}
    \includegraphics[width=\textwidth]{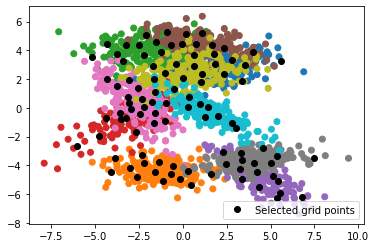}
\caption{The supremum distance selection.}
  \label{fig:sup3}
\end{subfigure}
\hfill
\begin{subfigure}{0.45\textwidth}
    \includegraphics[width=\textwidth]{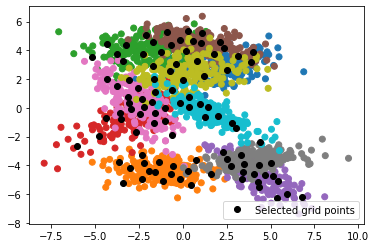}
 \caption{The ITD selection.}

  \label{fig:kernel3}
\end{subfigure}
\vspace{-2ex}

\caption{Gaussian Mixture model with 10 centers and samples of 200 drawn from each center; $\text{dim}(\beta) = 1000000$, $\text{dim}(\gamma) = 500$, and 100 selected  representative points.}
\label{fig:P2}
\end{figure}

\vspace{-5ex}

\begin{figure}[H]
\centering

\begin{subfigure}{0.45\textwidth}
    \includegraphics[width=\textwidth]{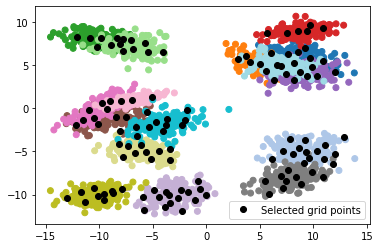}
\caption{The supremum distance selection.}
  \label{fig:sup1}
\end{subfigure}
\hfill
\begin{subfigure}{0.45\textwidth}
    \includegraphics[width=\textwidth]{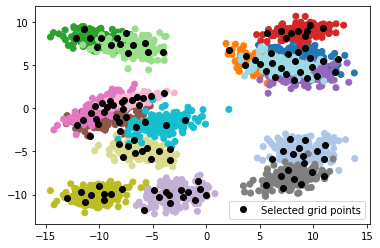}
 \caption{The ITD selection.}
  \label{fig:kernel1}
\end{subfigure}
\vspace{-2ex}

\caption{Gaussian Mixture model with 16 centers and samples of 100 drawn from each center; $\text{dim}(\beta) = 1638400$, $\text{dim}(\gamma) = 640$, and 128 selected  representative points.}
\label{fig:P3}
\end{figure}

The data used for sampling in the six examples listed in Table \ref{gmm} is as follows, where $\mu_i$ and $\sigma_i$ represent the mean and covariance matrix of center $i$, and $\lambda_0$ is the weight vector.

\begin{description}

\item[Dimension 2 and 5 centers]

$$\mu_1=\begin{bmatrix}
0.31266704 \\
0.27504179
\end{bmatrix}, \quad \mu_2=\begin{bmatrix}
0.15120579 \\
-0.92187417
\end{bmatrix}, \quad \mu_3=\begin{bmatrix}
-0.28437279 \\
0.89136637
\end{bmatrix},$$
$$
\mu_4=\begin{bmatrix}
-0.87991064\\
0.72808421
\end{bmatrix},\quad \mu_5=\begin{bmatrix}
0.75458105 \\
-0.89761267
\end{bmatrix}$$

$$\sigma_1=\begin{bmatrix}
2.27171261 & -0.19234173 \\
-0.19234173 & 0.30900127
\end{bmatrix},\quad \sigma_2=\begin{bmatrix}
2.72538843 & -0.18819093 \\
-0.18819093 & 0.76035309
\end{bmatrix},$$
$$\sigma_3=\begin{bmatrix}
2.28960495 & -0.00992103 \\
-0.00992103 & 0.26865532
\end{bmatrix},$$
$$\sigma_4=\begin{bmatrix}
1.52649829 & 1.0452592 \\
1.0452592 & 1.74912699
\end{bmatrix}, \quad \sigma_5=\begin{bmatrix}
2.41616602 & 0.48373093 \\
0.48373093 & 0.51568926
\end{bmatrix}.$$

$$\lambda_0 = [0.15743525, 0.28348483, 0.10232679, 0.03627818, 0.42047495].$$

\item[Dimension 2 and 10 centers]

$$\mu_1=\begin{bmatrix}
 1.79229996  \\
 3.03739036
\end{bmatrix}, \quad \mu_2=\begin{bmatrix}
-1.19058867 \\
-4.34063653
\end{bmatrix}, \quad \mu_3=\begin{bmatrix}
-2.11854401  \\
4.09593528
\end{bmatrix},$$

$$\mu_4=\begin{bmatrix}
 -2.86614646 \\
 -0.47876038
\end{bmatrix}, \quad \mu_5=\begin{bmatrix}
 4.3120602  \\
 -4.75100772
\end{bmatrix}, \quad \mu_6=\begin{bmatrix}
1.00548917  \\
4.501295
\end{bmatrix},$$

$$\mu_7=\begin{bmatrix}
 -2.69697121  \\
 0.48489919
\end{bmatrix}, \quad \mu_8=\begin{bmatrix}
 4.09128375 \\
 -3.66830554
\end{bmatrix}, \quad \mu_9=\begin{bmatrix}
0.23412581  \\
2.50409859
\end{bmatrix},$$

$$\mu_{10}=\begin{bmatrix}
1.69013241 \\
-0.3224714
\end{bmatrix}$$

$$\sigma_1=\begin{bmatrix}
2.59847307 & 0.25010595 \\
0.25010595 & 0.60533531
\end{bmatrix},\quad \sigma_2=\begin{bmatrix}
2.53216876 & -0.14189311 \\
-0.14189311 &  0.3939682
\end{bmatrix}$$

$$\sigma_1=\begin{bmatrix}
2.59847307 & 0.25010595 \\
0.25010595 & 0.60533531
\end{bmatrix},\quad \sigma_2=\begin{bmatrix}
2.53216876 & -0.14189311 \\
-0.14189311 &  0.3939682
\end{bmatrix},$$

$$\sigma_3=\begin{bmatrix}
2.5064178 & 0.23770523 \\
0.23770523 & 0.49430945
\end{bmatrix},\quad \sigma_4=\begin{bmatrix}
1.91557475 & 0.98748 \\
0.98748   & 1.24310123
\end{bmatrix},$$

$$\sigma_5=\begin{bmatrix}
1.75850071 & -0.95310038 \\
-0.95310038 &  0.93108462
\end{bmatrix},\quad \sigma_6=\begin{bmatrix}
2.21886025 & -0.58325887 \\
-0.58325887 &  0.46024507
\end{bmatrix},$$

$$\sigma_7=\begin{bmatrix}
1.35681743 & -1.09247798 \\
-1.09247798 &  1.61345614
\end{bmatrix},\quad \sigma_8=\begin{bmatrix}
2.20271214 & 0.32602845\\
0.32602845 & 0.29813593
\end{bmatrix},$$

$$\sigma_9=\begin{bmatrix}
2.5345792 & 0.3575912 \\
0.3575912 & 0.63866598
\end{bmatrix},\quad \sigma_{10}=\begin{bmatrix}
1.7025472 & -0.87512782\\
-0.87512782 &  0.69863228
\end{bmatrix}.$$

$$\lambda_0 = [0.01564952, 0.15994791, 0.08991018, 0.14837025, 0.2005688,\\$$
$$0.11043623, 0.10277117, 0.01477645, 0.05505221, 0.10251729].$$

\item[Dimension 2 and 16 centers]

$$\mu_1=\begin{bmatrix}
9  \\
6
\end{bmatrix}, \quad \mu_2=\begin{bmatrix}
9\\
-5
\end{bmatrix}, \quad \mu_3=\begin{bmatrix}
5 \\
5
\end{bmatrix},\quad \mu_4=\begin{bmatrix}
-10 \\
8
\end{bmatrix},$$

$$\mu_5=\begin{bmatrix}
-7  \\
7
\end{bmatrix}, \quad \mu_6=\begin{bmatrix}
9\\
9
\end{bmatrix}, \quad \mu_7=\begin{bmatrix}
9 \\
4
\end{bmatrix},\quad \mu_8=\begin{bmatrix}
-3 \\
-10
\end{bmatrix},$$

$$\mu_9=\begin{bmatrix}
-9  \\
-1
\end{bmatrix}, \quad \mu_{10}=\begin{bmatrix}
-10\\
0
\end{bmatrix}, \quad \mu_{11}=\begin{bmatrix}
-7 \\
1
\end{bmatrix},\quad \mu_{12}=\begin{bmatrix}
8 \\
-8
\end{bmatrix},$$

$$\mu_{13}=\begin{bmatrix}
-10  \\
-10
\end{bmatrix}, \quad \mu_{14}=\begin{bmatrix}
-6\\
-5
\end{bmatrix}, \quad \mu_{15}=\begin{bmatrix}
-4 \\
-2
\end{bmatrix},\quad \mu_{16}=\begin{bmatrix}
7 \\
5
\end{bmatrix}$$

$$\sigma_1=\begin{bmatrix}
2.49118778 & -0.52372617 \\
-0.52372617 &  0.57790164
\end{bmatrix},\quad \sigma_2=\begin{bmatrix}
2.7960884 & 0.07202756\\
0.07202756 & 0.7859871
\end{bmatrix},$$

$$\sigma_3=\begin{bmatrix}
1.84191462 & -0.85412423 \\
-0.85412423 &  0.7311887
\end{bmatrix},\quad \sigma_4=\begin{bmatrix}
2.47150611 &-0.19202043\\
-0.19202043 &  0.36136642
\end{bmatrix},$$

$$\sigma_5=\begin{bmatrix}
2.57147619 & -0.16840333 \\
-0.16840333 &  0.5958216
\end{bmatrix},\quad \sigma_6=\begin{bmatrix}
2.36451515 & 0.25765723\\
0.25765723 & 0.34956037
\end{bmatrix},$$

$$\sigma_7=\begin{bmatrix}
2.70927746& 0.1068342 \\
0.1068342 & 0.62088381
\end{bmatrix},\quad \sigma_8=\begin{bmatrix}
2.70676862 & 0.3036051\\
0.3036051 & 0.80040108
\end{bmatrix},$$

$$\sigma_9=\begin{bmatrix}
2.8751296 & 0.01771981 \\
0.01771981 & 0.86839542
\end{bmatrix},\quad \sigma_{10}=\begin{bmatrix}
1.66592325& 0.95932126\\
0.95932126&1.09707757
\end{bmatrix},$$

$$\sigma_{11}=\begin{bmatrix}
2.00518301& 0.44538624 \\
0.44538624& 0.20326645
\end{bmatrix},\quad \sigma_{12}=\begin{bmatrix}
2.42869787& 0.45667907\\
0.45667907& 0.60163072
\end{bmatrix},$$

$$\sigma_{13}=\begin{bmatrix}
2.46641227& 0.25352207 \\
0.25352207& 0.5117985
\end{bmatrix},\quad \sigma_{14}=\begin{bmatrix}
2.41577764&-0.39180762\\
-0.39180762&  0.55948601
\end{bmatrix},$$

$$\sigma_{15}=\begin{bmatrix}
2.51260059& 0.54101916 \\
0.54101916& 0.76445211
\end{bmatrix},\quad \sigma_{16}=\begin{bmatrix}
2.36191609& -0.34632195\\
0.34632195&  0.4390211
\end{bmatrix}.$$

$$\lambda_0  = [ 0.06011042, 0.07833323, 0.06601944, 0.05967994, 0.04640204, $$
$$0.07074346, 0.04792803, 0.09767407, 0.10554801. 0.04199756, $$
$$0.08671602, 0.05792878, 0.06221676, 0.10137871, 0.00778043, 0.00954309 ]. $$

\item[Dimension 3 and 3 centers]

$$\mu_1=\begin{bmatrix}
0.10827605  \\
3.92946954  \\
3.96293089
\end{bmatrix}, \quad \mu_2=\begin{bmatrix}
-3.7441469 \\
-2.92757122 \\
-4.48532797
\end{bmatrix}, \quad \mu_3=\begin{bmatrix}
-0.59190156 \\
-4.70123789 \\
-0.43166776
\end{bmatrix}.$$

$$\sigma_1=\begin{bmatrix}
1.68353569 &  1.50050598 &  0.0679262  \\
1.50050598 &  2.16517974 & -0.15853728 \\
0.0679262 & -0.15853728 &  0.44174394
\end{bmatrix},$$
$$\sigma_2=\begin{bmatrix}
1.44858553 & -1.2905356 & -0.80287245\\
-1.2905356 &  2.04023659 &  0.78596847 \\
-0.80287245 &  0.78596847 & 1.23990577
\end{bmatrix},$$

$$\sigma_3=\begin{bmatrix}
0.54892866 & -0.2603193 & -0.0255528  \\
-0.2603193 &  3.106415  &  0.92873064 \\
-0.02555282 &  0.92873064 &  0.63110853
\end{bmatrix}.$$

$$ \lambda_0 = [0.35538777, 0.45691364, 0.18769858].$$

\item[Dimension 3  and 5 centers]

$$\mu_1=\begin{bmatrix}
1.56333522 \\
1.37520896  \\
0.75602894
\end{bmatrix}, \quad \mu_2=\begin{bmatrix}
-4.60937084 \\
-1.42186396 \\
4.45683187
\end{bmatrix}, \quad \mu_3=\begin{bmatrix}
-4.3995532 \\
3.64042104 \\
3.77290526
\end{bmatrix},$$

$$\mu_4=\begin{bmatrix}
-4.48806334 \\
1.52418615 \\
0.51751369
\end{bmatrix}, \quad \mu_5=\begin{bmatrix}
0.97513253 \\
-0.16471376 \\
-2.17011839
\end{bmatrix}.$$

$$\sigma_1=\begin{bmatrix}
0.6244727 & -0.7239226 & -0.4573932  \\
-0.7239226 & 1.85830256 & 1.37237297 \\
-0.4573932 & 1.37237297 & 1.83967989
\end{bmatrix},$$
$$\sigma_2=\begin{bmatrix}
1.52528618 & 1.13821566 & -0.96273147\\
1.13821566 &  2.07065186 & -0.85116858 \\
-0.96273147 & -0.85116858 &  1.24867171
\end{bmatrix},$$
$$\sigma_3=\begin{bmatrix}
1.83612253 & 1.4455393 & -0.62284455  \\
1.4455393 & 2.04441818 & -0.51266313\\
-0.62284455 & -0.51266313 & 0.95234161
\end{bmatrix},$$
$$\sigma_4=\begin{bmatrix}
 0.72600504 & -0.6691261 & -0.66730684 \\
-0.6691261 & 2.12310268 & 1.28413393\\
-0.66730684 &  1.28413393 &  1.60606812
\end{bmatrix},$$
$$\sigma_5=\begin{bmatrix}
0.20310482 & -0.36713897 &  0.03372048  \\
-0.36713897 &  2.862603 & -0.76036188\\
0.03372048 & -0.76036188 &  0.26176859]
\end{bmatrix}.$$

$$\lambda_0 = [0.15743525, 0.28348483, 0.10232679, 0.03627818, 0.42047495].$$

\item[Dimension 5 and 3 centers]

$$\mu_1=\begin{bmatrix}
0.10827605\\
3.92946954\\
3.96293089\\
-3.7441469\\
-2.92757122
\end{bmatrix}, \quad \mu_2=\begin{bmatrix}
-4.48532797\\
-0.59190156\\
-4.70123789\\
-0.43166776\\
1.49144048
\end{bmatrix}, \quad \mu_3=\begin{bmatrix}
-2.21512717\\
1.76254902\\
0.90862817\\
-4.76018118\\
0.58854088
\end{bmatrix}.$$

$$\sigma_1=\begin{bmatrix}
0.6141946& -0.0185777&  0.1488841& -0.0632101& -0.2064803 \\
-0.0185777&  1.4940510&  1.4099103&  0.8157628& -1.1855394 \\
0.1488841& 1.4099103&  2.2648389&  1.0221659& -1.7608476 \\
-0.0632101&  0.8157628&  1.0221659&  1.1607710& -0.9997849 \\
-0.2064803& -1.1855394& -1.7608476& -0.9997849&  1.9458983
\end{bmatrix},$$
$$\sigma_2=\begin{bmatrix}
1.7200913&  1.3976971&  1.1416574& 0.0387295& -1.5148418 \\
 1.3976971&  1.612375&  0.969004& -0.025671& -1.441501  \\
1.1416574& 0.9690040&  1.1830867& -0.1298122& -1.0327085 \\
0.0387295& -0.0256711& -0.1298122&  0.7039679&  0.0605373 \\
-1.5148418& -1.441501 & -1.0327085&  0.0605373&  1.9484404
\end{bmatrix},$$
$$\sigma_3=\begin{bmatrix}
0.5592179& -0.1198072& -0.1764103&  0.5732007&  0.1635879 \\
-0.1198072&  1.3423229&  0.3395780& -1.7867579& -0.3338649 \\
-0.1764103&  0.3395780&  0.8870684& -0.6177889&  0.0123156 \\
0.5732007& -1.7867579& -0.6177889&  4.258851 &  1.1801519\\
 0.1635879& -0.3338649&  0.0123156&  1.1801519&  0.8692625
\end{bmatrix}.$$

$$\lambda_0 = [0.35538777,0.45691364,0.18769858].$$

\end{description}

\end{document}